%% file: main.tex
\pdfoutput=1

\DeclareSymbolFont{AMSb}{U}{msb}{m}{n}
\documentclass[11pt,a4paper,leqno,noamsfonts]{amsart}

   \makeatletter
   \renewcommand\@biblabel[1]{#1.}
   \makeatother
   \usepackage{multirow}

\usepackage{graphicx}
\usepackage{subcaption}
\usepackage{pythontex}
\usepackage{amssymb}

\usepackage{listings} 

\usepackage[bookmarks=false]{hyperref}

\newtheorem{manualtheoreminner}{Theorem}

\newenvironment{manualtheorem}[1]{%
  \IfBlankTF{#1}
    {\renewcommand{\themanualtheoreminner}{\unskip}}
    {\renewcommand\themanualtheoreminner{#1}}%
  \manualtheoreminner
}{\endmanualtheoreminner}

\newtheorem{manualconjectureinner}{Conjecture}
\newenvironment{manualconjecture}[1]{%
  \IfBlankTF{#1}
    {\renewcommand{\themanualconjectureinner}{\unskip}}
    {\renewcommand\themanualconjectureinner{#1}}%
  \manualconjectureinner
}{\endmanualconjectureinner}

\usepackage{listings}
\usepackage{color}

\definecolor{dkgreen}{rgb}{0,0.6,0}
\definecolor{gray}{rgb}{0.5,0.5,0.5}
\definecolor{mauve}{rgb}{0.58,0,0.82}

\usepackage{mathtools}

\usepackage[english]{babel}
\usepackage[dvipsnames]{xcolor}
\usepackage{graphicx,pifont,soul,physics} 
\usepackage[utopia]{mathdesign}
\usepackage[a4paper,top=3cm,bottom=3cm,
           left=3cm,right=3cm,marginparwidth=60pt]{geometry}

\usepackage{fancyhdr}
\usepackage{float}

\usepackage[utf8]{inputenc}
\usepackage{braket,caption,comment,mathtools,stmaryrd}
\usepackage[usestackEOL]{stackengine}
\usepackage{multirow,booktabs,microtype,relsize}

\usepackage[bookmarks=false]{hyperref}%
      \hypersetup{colorlinks,%
           citecolor=britishracinggreen,%
            filecolor=black,%
            linkcolor=cobalt,%
            urlcolor=cornellred}
      \numberwithin{equation}{section}
\usepackage[capitalise]{cleveref}
\input{macros.tex}
\DeclareMathAlphabet\BCal{OMS}{cmsy}{b}{n}

\address{Centre for Mathematical Sciences, University of Cambridge, Wilberforce Road, CB3 0WA, Cambridge, United Kingdom}

\title[An obstruction to the maximal singularity]{An obstruction to the maximal singularity of the Hilbert scheme of points on threefolds}

\author{Fatemeh Rezaee}
\email{fr414@cam.ac.uk}
\author{Parsa Saberi}
\email{ps2046@cam.ac.uk}

\begin{document}
\maketitle
\begin{abstract} 
We prove a necessary condition conjecture (\cite[Conjecture B]{Rezaee-23-Conjectures}) for the maximal singularity of the Hilbert scheme of points in 3D for a large class of degrees, which in particular implies the 1978 Brian\c{c}on-Iarrobino Conjecture for a tetrahedral degree, and vastly generalises the result in \cite{MR1}.  The novelty lies in introducing a new upper bound on the dimensions of the tangent spaces of the claimed non-maximally singular ideals, constructing canonical ideals that satisfy the claimed maximal singularity condition, providing a lower bound on the dimension of their tangent spaces, and finally comparing the two cases.
\end{abstract}

{\hypersetup{linkcolor=black}
\tableofcontents}

\section{Introduction}

Moduli spaces of sheaves are central in algebraic geometry. Among them, Hilbert schemes\footnote{In the introduction, Hilbert schemes can have positive Hilbert polynomials, while for the following sections, we only focus on Hilbert schemes of points.} are one of the most significant ones, and they parametrise subschemes of a certain type in an ambient scheme. Hilbert schemes of points have a very natural definition, and they are ubiquitous in mathematics, from algebraic geometry to enumerative geometry to combinatorial geometry to representation theory and commutative algebra. By definition, for a fixed positive integer $l$, they parametrise zero-dimensional closed subschemes of degree $l$ of an ambient variety $X$, denoted by $\Hilb^l(X)$. The complete geometry of $\Hilb^l(X)$, even when the ambient space is simply an affine space of dimension at least three, is very complex and mysterious, and there are various open problems regarding it. The case where $X$ is a threefold is of particular importance in enumerative geometry, as, for example, $\Hilb^l(\BA^3)$ is known to be the critical locus of a super-potential on a smooth variety (\cite{BBS}). Also, when $X$ is a Calabi-Yau threefold, the Hilbert scheme becomes important in enumerative geometry and curve counting theories. There is a Murphy’s Law for Hilbert schemes, which roughly says that they can have arbitrarily bad singularities that one can imagine, confirming the complexity of these spaces (\cite{HM}). Murphy's law for Hilbert schemes was proved in \cite{Vakil06} for positive degree Hilbert polynomials, and in \cite{Jelisiejew20} for degree zero Hilbert polynomials for the ambient space of dimension at least $16$, up to retraction. The Hilbert scheme of points on threefolds was studied in \cite{Ramkumar-Sammartano}, \cite{Jelisiejew-Ramkumar-Sammartano24}, \cite{HuX}, and \cite{KatzS}. For smoothness on the Hilbert scheme of pints on threefolds via broken Gorenstein structure for finite schemes, see \cite{Ramkumar-Sammartano2}. Local equations for Hilbert schemes have been studied in \cite{Ilten-Meazzini-Petracci}.  Related to threefolds, the parity conjecture \cite{PandharipandeSlides} was proved for monomial ideals in \cite{Ramkumar-Sammartano1} and disproved in \cite{GGGL23} for non-monomial ideals.  The reducibility problem of the Hilbert scheme of points was studied in \cite{Iarrobino72}, \cite{Hilb8}, \cite{fogarty}, \cite{Jardim} and  \cite{Hilb_11}. There are also various works on the enumerative geometry concerning Hilbert schemes including \cite{BBS}, \cite{BFHilb}, \cite{JKSCounterBehrend}, \cite{MNOP1}, \cite{Nesterov25}, \cite{PT}, and \cite{RicolfiSign}. Recently, classification of a generalised version of Hilbert schemes, Quot schemes, has been considered in \cite{Skjelnes-Smith-Stillman2026}. 

 To measure how badly-behaved the singularities could be, one can consider the dimension of the tangent space to the Hilbert scheme at each point. There are several prior conjectures and counterexamples regarding the maximal singularity; for example, a couple of conjectures were made in \cite{Bri-Iar} and \cite{Sturmfels}, and the counterexamples were given in \cite{Sturmfels} and \cite{Ramkumar-Sammartano}, respectively. Also, recently, a further conjecture on the maximal singularity has been given in \cite{ARZConj2024}.

 For simplicity, instead of considering $X$, we can restrict ourselves to $\BA^N$, since the singularity problem is a local one. One of the longest-standing conjectures regarding the maximal singularities of  $\Hilb^l(X)$ is the following, which predicted the unique point of the Hilbert scheme at which the maximum singularity is attained, for the case of a tetrahedral degree $l$.

\begin{manualconjecture}{A}[\cite{Bri-Iar}] \label{1978Conj} The most singular point of $\text{\:}\Hilb^{{N-1+k \choose N}}(\BA^N)$ corresponds to the ideal
 $\mathfrak{m}^k$, where $\mathfrak{m}$ is the maximal ideal in $\BC[x_1,x_2,\dots,x_N]$.
\end{manualconjecture}

This was generalised to arbitrary degree as a necessary condition for maximal singularity:

\begin{manualconjecture}{B}[Necessary condition for maximal singularity]  \cite[Conjecture B]{Rezaee-23-Conjectures}
\label{necessaryCondition} Let $R=\BC[x_1,x_2,\ldots,x_N]$. For $N\geq 3$,   let $I$ be a zero-dimensional Borel-fixed ideal of degree $l$ in $R$  given by
   \begin{align*}
       I=(x_1^{m_1},x_2^{m_2}, \ldots,x_N^{m_N},\text{mixed monomial generators}),
   \end{align*}
   where  $m_1\leq m_2\leq\ldots\leq m_N$. 
   If ${N-1+k \choose N}\leq l<{N+k\choose N}$, and $I$ corresponds to a point of $\text{\:}\mathrm{Hilb}^{l}(\BA^N)$ of maximal singularity, then $m_1=k$.
\end{manualconjecture}

The purpose of this paper is to prove this conjecture for $N=3$, for the following classes of degrees:
\begin{align}\label{Type 1}
l=&{k+2\choose 3}+j,  \text{ for $0\leq j\leq (k+3)/2$},
    \\\label{Type 2}
      l=&{k+3\choose 3}-j,  \text{ for $1\leq j\leq 4$},\\\label{Type 3}
 l=&{k+3\choose 3}-{b\choose 2}, \text{ for $b=2,3,4,k+1,k+2$}.
\end{align}

Note that there are the following overlaps: (1) $j=0$ in \eqref{Type 1}
 with $b=k+2$ in \eqref{Type 3}, (2) $j=1$ in \eqref{Type 2} with $b=2$ in \eqref{Type 3}, and (3)  $j=3$ in \eqref{Type 2} with $b=3$ in \eqref{Type 3}.

\begin{manualtheorem}{B}[Theorem \ref{Cor: ConjB}]  
\label{necessaryConditionForTypes123} Let $N=3$. For degrees \eqref{Type 1}-\eqref{Type 3}, Conjecture \ref{necessaryCondition} holds.
\end{manualtheorem}

An immediate corollary of this is the following result, which was first proved in \cite{MR1}.

\begin{manualtheorem}{A}[Corollary \ref{Cor:BIConj3D}]  
\label{BI3D}  For $N=3$, the Brian\c{c}on-Iarrobino Conjecture (Conjecture \ref{1978Conj})  holds.
\end{manualtheorem}

\subsection*{Strategy of the proof and novelty} Fix positive integers $l,m_1,k$ such that $m_1<k$, and $l$ is of types \ref{Type 1}-\ref{Type 3}. For a given zero-dimensional Borel-fixed ideal $I$ in $\BC[x,y,z]$ of degree $l$ with $m_1(I)=m_1$, we first provide an upper bound, $U_{m_1}(I)$, on the dimension of the tangent space, $T(I)$. Then, for the fixed $l$, we give a weaker upper bound $W_{m_1}(l)$ on $U_{m_1}(I)$, for any zero-dimensional Borel-fixed $I$ of degree $l$ with $m_1(I)=m_1<k$. Then, we construct a canonical ideal $\FJ_l$ of degree $l$ and $m_1(\FJ_l)=k$, and find a lower bound $V(l)$ on $T(\FJ_l)$. We also prove that $V(l)>W_{m_1}(l)$, for $m_1<k$. Therefore, overall, for a given zero-dimensional Borel-fixed ideal $I$ in $\BC[x,y,z]$ of degree $l$ with $m_1(I)=m_1<k$, we have
\begin{align*}
T(I)\leq U_{m_1}(I)\leq W_{m_1}(l)<V(l)\leq T(\FJ_l),\end{align*}
where $\FJ_l$ is a zero-dimensional Borel-fixed ideal of degree $l$ and $m_1(\FJ_l)=k$. 

Here, the novelty lies in introducing a new upper bound $U_{m_1}(I)$, constructing canonical ideals $\FJ_l$ with $m_1(\FJ_l)=k$, and finding a lower bound $V(l)$ on the dimension of the tangent space at this point (which is indeed equal to $T(\FJ_l)$), which is sufficient to prove the necessary condition conjecture for a large class of degrees. 

The method is expected to be extended to higher dimensions and other degrees, which can shed light on future studies of the Hilbert scheme singularities.

\subsection*{Acknowledgements} This is an outcome of the King's College (Cambridge) Summer Research Programme 2026; we thank the administrative support of King's College for coordination of the programme. The authors are grateful to Mark Gross for his support, and to Tony Iarrobino and Alessio Sammartano for helpful comments. FR was supported by the UKRI grant EP/X032779/1, and PS was supported by Global Talent Fund. We acknowledge the use of Macaulay2 \cite{GS} and Python.

\subsection*{Conventions} We define $\BN=\{0,1,2,3,\ldots\}$. For a zero-dimensional ideal $I$ of degree $l$, we denote by $[I]$ the corresponding point in $\Hilb^l(\BA^3)$. Throughout the paper, when we use the term `component', we basically mean a connected component. Also, by a generator of a monomial ideal $I$, we mean an element of the minimal set of generators of $I$.  We use the variables $x, y, z$ for the 3D space, unless explicitly stated otherwise. Furthermore, we use the directional terms front, back, right, left, up, and down to refer to the positive and negative directions of $x$, $y$, and $z$, respectively.

\subsection*{Notation} Let $R=\BC[x,y,z]$ be the polynomial ring, and $I$ a zero-dimensional Borel-fixed ideal in $R$. 
\begin{center}
    
         \begin{tabular}{ r l } 

                   $\mathrm{m}$:& The zero-dimensional monomial ideal of degree one generated by the variables,\\& i.e., $\mathrm{m}=(x,y,z)$.\\

          $m_i(J)$:& For an ideal $J$ and $i=1,2,3$, we define $m_i(J)$ to be the exponent of the generators\\& purely in $x,y,z$, respectively, i.e., we can write\\& $J=(x^{m_1(J)},y^{m_2(J)},z^{m_3(J)},\text{mixed monomial generators})$\\

           $h(i,j)$:&  The \emph{height} of the $x=i$, $y=j$ column of $\BN^3\setminus \tilde I$ (see Definition \ref{def: height}).\\

          $l=l(I)$:& The degree of $I$, defined as $\mathrm{hom}(R,R/{I})$.\\

          $T(I)$:& The dimension of the  tangent space to the Hilbert scheme at $[I]$, which is\\&defined as $\mathrm{hom}(I,R/{I})$.\\

   $\FC(I)$:& The hotspot of $\BN^3\setminus \tilde I$ (Definition \ref{def: Hotspot}).\\

    $\FJ_l$:& The  canonical ideal of degree $l$ (Definition \ref{Def: canonicalConfig}).\\

           \end{tabular}
     \end{center}

     \begin{center}
    
         \begin{tabular}{ r l }

  $U_{m_1}(I)$:& The  upper bound on the tangent space at $[I]$, where $I=(x^{m_1},y^{m_2},z^{m_3},$\\&$\text{mixed terms})$ is a zero-dimensional Borel-fixed ideal (Section \ref{Section: UpperBoundU}).\\

    $W_{m_1}(l)$:& An  upper bound on  $U_{m_1}(I)$, for any zero-dimensional Borel-fixed ideal\\& $I=(x^{m_1},y^{m_2},z^{m_3},\text{mixed monomial terms})$ of degree $l$ (see \eqref{eq: WeakUpperBound}).\\

  $V(l)$:& A lower bound on the dimension of the tangent space of the canonical ideal $\FJ_l$ \\&(Section \ref{Section: CanonicalConfig})
 \end{tabular}
     \end{center}

  \section{Background} In this section, we recall the definitions and tools we need in this paper.

First, we recall the key notion of Borel-fixedness: a monomial ideal $I$ in $\BC[x_1,x_2,x_3]$ is \emph{Borel-fixed} (or \emph{strongly stable}), if for any minimal generator $g$ of $I$, if $x_i|g$, then $x_j(g/x_i)$ is an element of $I$ for any $1\leq j <i$. In the rest of the paper, for simplicity, we use $x,y,z$ for $x_1,x_2,x_3$, respectively.

\subsection{Bounded connected components}   
We use the decomposition of the tangent space described in \cite{Ramkumar-Sammartano}. For a zero-dimensional Borel-fixed ideal $I$, let $\tilde I$ be the points in $\BZ^3$ that correspond to the exponents of the monomial elements of $I$, e.g., the point  $(5,2,3)$ corresponds to $x^5y^2z^3$.  We represent the points of $\BN^3\backslash \tilde I$ by cubes. See Figure \ref{Fig: Components}, left.

For a zero-dimensional Borel-fixed ideal $I$ in three variables $x,y,z$, we associate the \emph{$x$-filtration} $I=\oplus_{i=0}^{\infty}x^iJ_i$, for $J_i$ an ideal in $\BC[y,z]$. For $0\leq i\leq m_1 - 1$, we define $I_i = (\{x^i\}\times\BN^2)\backslash J_i$ as the slicing $I_i$ of $\BN^3\setminus \tilde I$ along the $x$-axis (see Figure \ref{Fig: HandS}, left). If $m_1(I)=m_1$, there are $m_1$ many slices (see Figure \ref{Fig: HandS}, left).

The following definition/proposition is crucial in the rest of this paper.

\begin{definition}\cite[Proposition 1.5]{Ramkumar-Sammartano} For a point $[I]$ in $\Hilb^l(\BA^3)$, the dimension of the tangent space to the Hilbert scheme at $[I]$ is the sum of the numbers of bounded connected components of $(\tilde I+\alpha) \setminus \tilde I$ for all vectors $\alpha$ in $\BZ^3$. (See Figure \ref{Fig: Components}, right).

\end{definition}

\begin{definition}[Size of a bounded component] We define the \textit{size of a bounded component} of $(\tilde I+\alpha)\setminus \tilde I$ as the number of unit cubes contained in the component.
\end{definition}

  \begin{figure}[H]

  \subcaptionbox*{}[.93\linewidth]{
    \includegraphics[width=\linewidth]{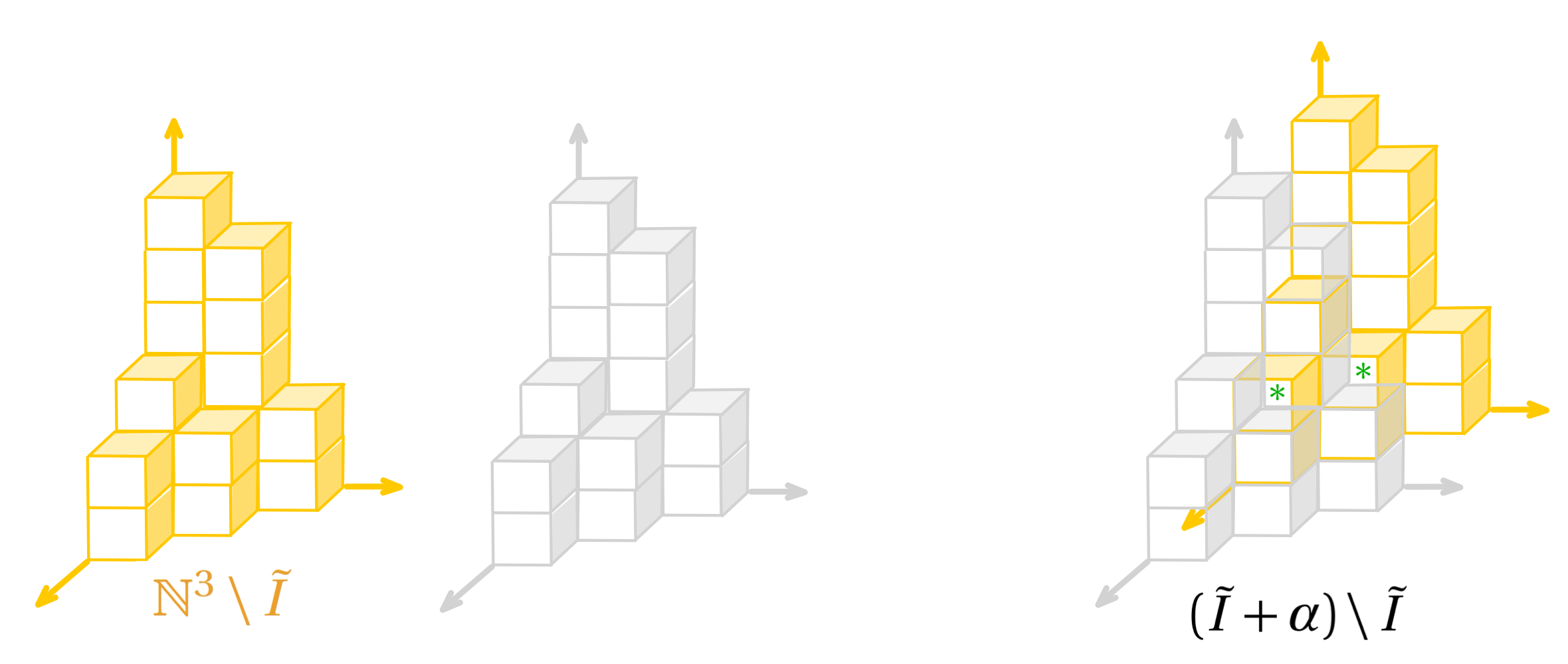}
  }
  \caption{We use a copy of  $\BN^3\setminus\tilde I$ in a different colour and move it along $\alpha$ to create $(\tilde I+\alpha)\setminus \tilde I$. In the right picture, for $\alpha=(1,-1,-1)$, the bounded components of $(\tilde I+\alpha)\setminus \tilde I$ are marked by $*$.}
  \label{Fig: Components}

\end{figure}

\begin{definition}[heigh]\label{def: height}
 For a zero-dimensional Borel-fixed ideal $I$,  we define  the \emph{height} $h(i,j)$ as follows:  for any $0\leq i<m_1(I)$ and $0\leq j<m_2(I_i)$, let $h(i, j)$ be the least positive integer such that $x^iy^jz^{h(i, j)}\in I$, i.e., $h(i, j)$ is the height of the $j$th column of $I_i$. See Figure \ref{Fig: TlessU}, left.
\end{definition}

\begin{definition}[Hotspot in $\BN^3\setminus \tilde I$]\label{def: Hotspot} For a zero-dimensional Borel-fixed ideal $I$ in three variables, we define the \emph{hotspot} $\FC(I)$ to be the set of the cubes in $\BN^3\setminus \tilde I$ which are immediately below a generator of $I$ (recall that, by our convention,  a generator is an element of the minimal set of generators). In other words,  $\FC(I)$ is the set of the top cubes of the columns of $\BN^3\setminus \tilde I$. See Figure \ref{Fig: Caravansarai}.

\begin{figure}[h]

  \subcaptionbox*{}[.94\linewidth]{
    \includegraphics[width=\linewidth]{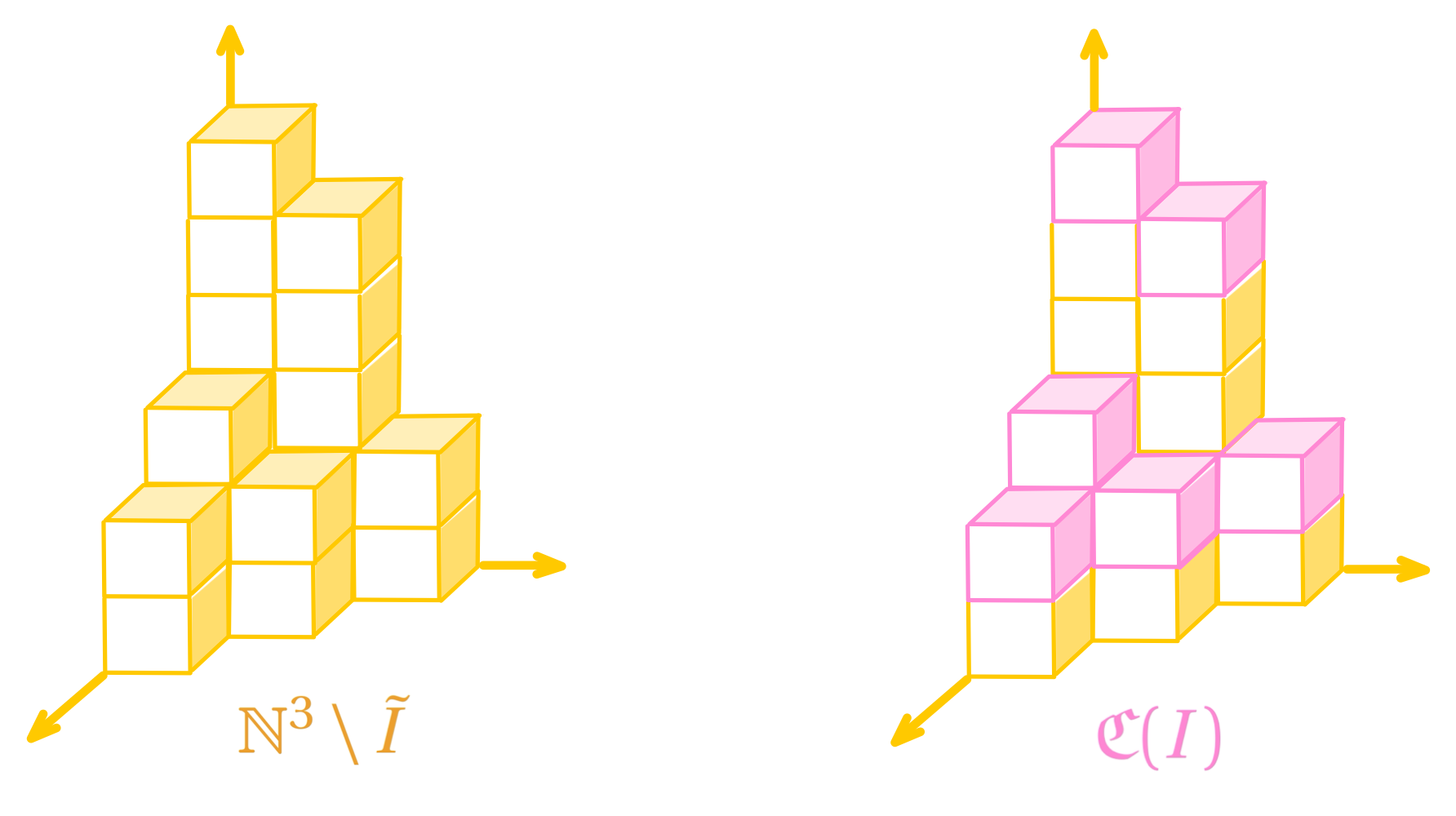}
  }
  \caption{$\BN^3\setminus \tilde I$ and the associated hotspot $\FC(I)$.}
  \label{Fig: Caravansarai}

\end{figure}
\end{definition}

\begin{definition}[Vertical distance of a cube in $\BN^3\setminus\tilde I$] For each cube in $\BN^3\setminus\tilde I$, we define its \emph{vertical distance} to be its distance from the generator above it.
\end{definition}

We now have a lemma, which later will be helpful in giving a lower bound on the tangent space.

\begin{lemma}\label{lem: Hotspot}Each vector $\alpha$ from a generator of $I$ to a cube in $\FC(I)$ represents a vector of the tangent space at $[I]$. Moreover, the size of the produced bounded component by such $\alpha$ is one.
\end{lemma}
\begin{proof}
Since the generators correspond to the corners of the configuration of $\BN^3\setminus\tilde I$, for all vectors $\alpha$ from an arbitrary generator $g$ to an element $c$ of $\FC(I)$, the corresponding component of $(\tilde I+\alpha)\setminus\tilde I$ must be bounded. This is due to the fact that, by the definition of $\FC(I)$, the vertical distance of $c$ is one; hence the size of the bounded component  of $(\tilde I+\alpha)\setminus\tilde I$ associated to it is exactly one. Also, since the sizes of the bounded components are one, all the components are different.
\end{proof}

Now, we have an immediate corollary.

\begin{corollary}[Lower bound on the tangent space]\label{Cor: LowerBoundOnT}
    Let $I$ be a zero-dimensional Borel-fixed ideal in three variables, and let $G(I)$ be the number of generators of $I$. Then, 
    \begin{align*}
        T(I)\geq G(I)\cdot \#\FC(I).
    \end{align*}
\end{corollary}
  \section{An upper bound on the tangent space}\label{Section: UpperBoundU} In this section, for a zero-dimensional Borel-fixed ideal $I$ of degree $l$ with $m_1(I)=m_1$, we define an upper bound $U_{m_1}(I)$ on the dimension of the tangent space to the Hilbert scheme at $[I]$.
  
  For a zero-dimensional Borel-fixed ideal $I$ of degree $l$ with $m_1(I)=m_1$, we define
  \begin{align*}
      U_{m_1}(I):=(2m_1+1)l-\sum_{0\leq j\leq i\leq m_1-1} H_{i,j}(I)-\sum_{0\leq j\leq i\leq m_1-1} S_{i,j}(I),
  \end{align*}
  where $H_{i,j}(I)$ is the size of the piece of the $j$th slice which is strictly above the highest cube of the $i$th slice, and $ S_{i,j}(I)$ is the size of the piece of the $j$th slice which is strictly to the right of the rightmost cubes of the $i$th slice, for $0\leq j \leq i\leq m_1-1$ (see Figure \ref{Fig: HandS}). Observe that both $H_{i,j}(I)$ and $S_{i,j}(I)$ are zero for $i = j$.

  \begin{figure}[H]

  \subcaptionbox*{}[.99\linewidth]{
    \includegraphics[width=\linewidth]{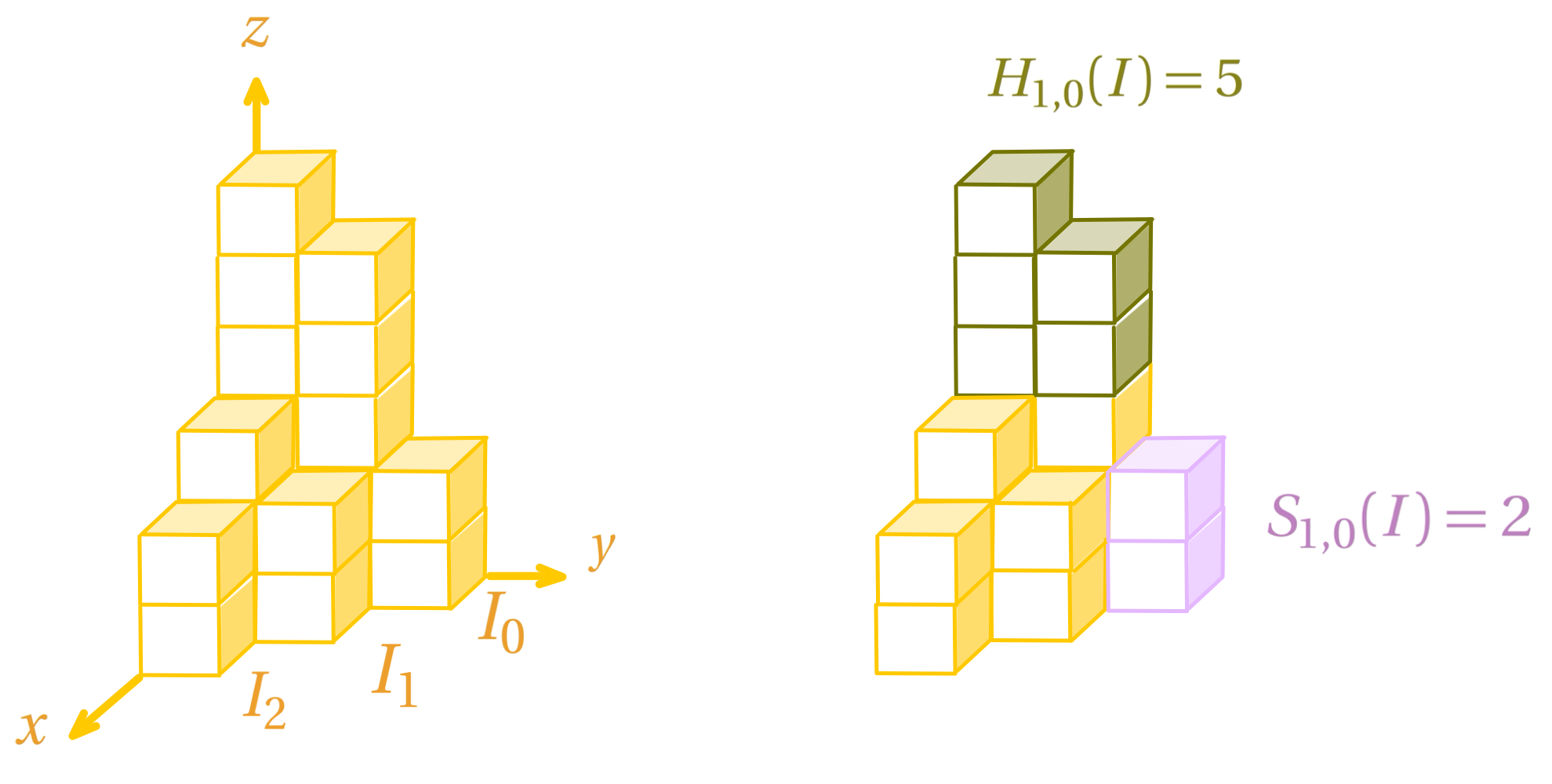}
  }
  \caption{$T(I)=98<U_3(I)=(2\cdot3+1)20-(1+5+7)-(2+2+7)=116$.}
  \label{Fig: HandS}

\end{figure}

\begin{lemma}[$U_{m_1}(I)$ is an upper bound on $T(I)$]\label{lem: TLessU}
    For a zero-dimensional Borel-fixed ideal $I$ with $m_1(I)=m_1$, we have $T(I)\leq U_{m_1}(I)$.
\end{lemma}
\begin{proof}

    For the vectors $\alpha = (\alpha_1, \alpha_2, \alpha_3)$  with $\alpha_1<-m_1$, we see that $(\tilde I + \alpha)\backslash\tilde I$ is connected and unbounded. For those with $m_1\leq\alpha_1$, we observe that $(\tilde I + \alpha)\backslash\tilde I$ yields either a single unbounded component or no component at all (the latter occurs when both $\alpha_2\leq 0$ and $\alpha_3\leq 0$). Therefore, we only have to consider $\alpha$'s with $-m_1\leq\alpha_1<m_1$. For a bounded component of $(\tilde I + \alpha)\backslash\tilde I$, define its backmost slice as the one with the least $x$-coordinate. For each $\alpha_1$, we give an upper bound on the number of  bounded components by counting their backmost slices, and show that these upper bounds add up to $U_{m_1}(I)$.

  First, consider the case $-m_1\leq\alpha_1<0$. The backmost slice must have $x$-coordinate between $0$ and $m_1 +\alpha_1$ inclusive, because any element of $(\tilde I + \alpha)\backslash\tilde I$ with negative $x$-coordinate is in the unbounded component, and any element with $x$-coordinate bigger than $m_1 +\alpha_1$ is either connected to a cube with smaller $x$-coordinate or is in the unbounded component; in both cases, it cannot be in the backmost slice. Observe that for the backmost slice to have  $x$-coordinate $m_1+\alpha_1$, both $\alpha_2$ and $\alpha_3$ must be non-negative. This is because if either is negative, the $x = m_1 + \alpha_1$ cut of $(\tilde I + \alpha)\backslash\tilde I$ would be connected and unbounded; this is due to the fact that the negativity of $\alpha_2$ or $\alpha_3$ causes the slice $I_{m_1 + \alpha_1}$ to connect to the unbounded component from the left or below, respectively. For each $0\leq i<m_1$ and $0\leq j<m_2(I_i)$, recall that $m_2(I_i)$ is the width of $I_i$, and $h(i, j)$ is the height of the $j$th column of $I_i$ (see Figure \ref{Fig: TlessU}). For a backmost slice to appear within $I_{m_1+\alpha_1}$, we must have $0\leq\alpha_2< m_2(I_{m_1+\alpha_1})$, since if $\alpha_2\geq m_2(I_{m_1+\alpha_1})$, then $(\tilde I + \alpha)\backslash\tilde I$ would have empty intersection with the $x = m_1+\alpha_1$ plane. For any $\alpha_2$ with $0\leq\alpha_2< m_2(I_{m_1+\alpha_1})$, for the backmost slice to be within $I_{m_1+\alpha_1}$, we must have $\max(0, h(m_1 + \alpha_1, \alpha_2)-h(m_1-1, 0))\leq\alpha_3<h(m_1 + \alpha_1, \alpha_2)$ for the following reason: (1) If $\max(0, h(m_1 + \alpha_1, \alpha_2)-h(m_1-1, 0))>\alpha_3$, then the top cube of the $\alpha_2$-th column of $m_2(I_{m_1 + \alpha_1})$ would be connected to the cube behind it if $\alpha_1 + m_1 > 0$, and would be connected to the unbounded component if $\alpha_1 + m_1 = 0$; thus, it would violate either being the backmost slice or being bounded. (2) If $\alpha_3\geq h(m_1 + \alpha_1, \alpha_2)$, then again $(\tilde I + \alpha)\backslash\tilde I$ would have empty intersection with the $x = m_1+\alpha_1$ plane. Hence, the number of the bounded components of $(\tilde I + \alpha)\backslash\tilde I$ with the backmost slice within $I_{m_1 + \alpha_1}$ has upper bound 
   \begin{align*}
      \sum_{\alpha_2 = 0}^{m_2(I_{m_1 + \alpha_1}) - 1} h(m_1 + \alpha_1, \alpha_2) - \max(0, h(m_1 + \alpha_1, \alpha_2)-h(m_1-1, 0)) = l_{m_1 + \alpha_1} - H_{m_1 - 1, m_1 + \alpha_1}(I),
    \end{align*}
    by definition of $H_{m_1 - 1, m_1 + \alpha_1}(I)$.

  \begin{figure}[h]

  \subcaptionbox*{}[.99\linewidth]{
    \includegraphics[width=\linewidth]{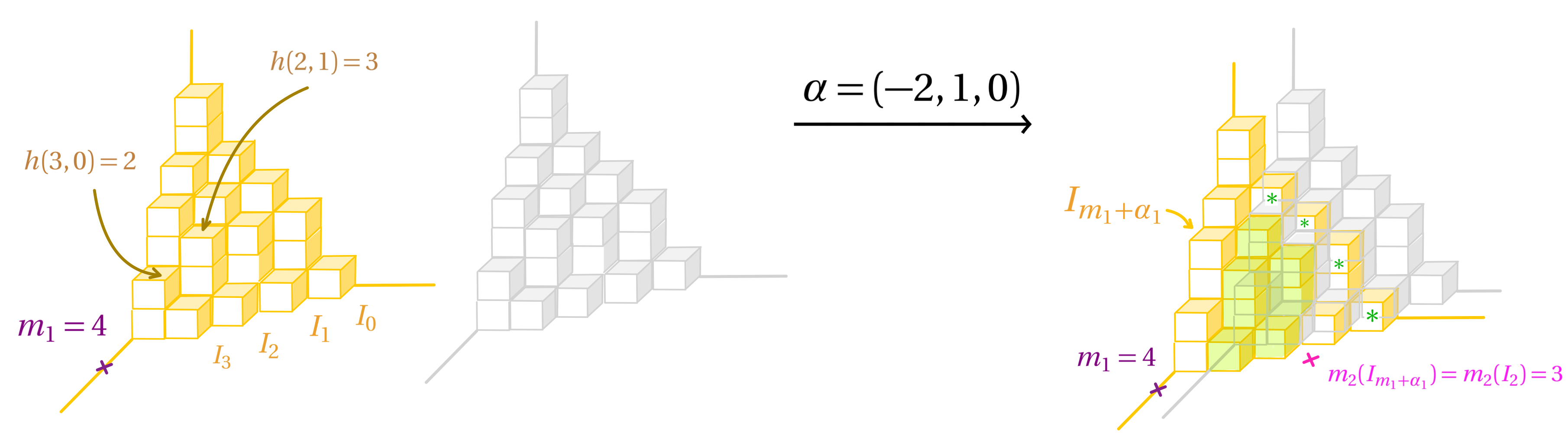}
  }
  \caption{Left: $\BN^3\setminus \tilde I$ and its copy to move along $\alpha$. Right: The bounded components of $(\tilde I + \alpha)\backslash\tilde I$. The one contained in more than one slice is highlighted; others are marked by a $*$.}
  \label{Fig: TlessU}

\end{figure}
    
    Now, for $0\leq i<m_1 + \alpha_1$, we give an upper bound on the number of bounded components of $(\tilde I + \alpha)\backslash\tilde I$ with the backmost slice within $I_i$. Since the backmost slice of a bounded component must be bounded-in-2D within its $x$-plane (in this case, the $x = i$ plane), we analyse the bounded-in-2D components. We show that the interaction of $I_i$ and $I_{i - \alpha_1} + \alpha$ produces $l_i + l_{i - \alpha_1}$ many bounded-in-2D components, from which at least $H_{i - \alpha_1 - 1, i}(I)$ many are connected to either a bounded component in the previous $x$-slice or the unbounded component; so the upper bound $l_i + l_{i - \alpha_1} - H_{i - \alpha_1 - 1, i}(I)$ will be obtained. To show that the number of the bounded-in-2D components is $l_i + l_{i - \alpha_1}$, we consider two cases $\alpha_2<0$ and $\alpha_2\geq0$.

    \begin{itemize}
         \item[(I)] $\alpha_2<0$: In this case, we must have $\alpha_2\geq-m_2(I_{i-\alpha_1})$; otherwise, the unbounded component would fill the space between $I_i$ and $I_{i - \alpha_1} + \alpha$, and therefore no bounded component will be produced. We claim that for each $0>\alpha_2\geq-m_2(I_{i-\alpha_1})$, there are $h(i - \alpha_1, -\alpha_2-1)$ many components. Observe that $I_i$ is strictly to the right of the $(-\alpha_2 - 1)$-th column of $I_{i - \alpha_1} + \alpha$, and therefore no bounded-in-2D component is (partially or completely) above the $(-\alpha_2 - 1)$-th column of $I_{i - \alpha_1} + \alpha$. This is due to the fact that otherwise, it would be connected to the unbounded component on the left, via the element of the unbounded component that is on the top of the $(-\alpha_2 - 1)$-th column of $I_{i - \alpha_1} + \alpha$ (see Figure \ref{Fig:TlessU1}, (1)). Hence, the top cube of a bounded-in-2D component must not exceed the height of this column. It should be below this column by a distance between $0$ and $h(i - \alpha_1, -\alpha_2 - 1) - 1$, inclusive (it cannot be completely below $I_{i - \alpha_1} + \alpha$, as otherwise it would not be in $\tilde I + \alpha$). We claim that for each of these distances, there is exactly one such component. The top cube of each bounded-in-2D component (which is unique due to Borel-fixedness) is immediately to the right of some cube of $I_{i - \alpha_1} + \alpha$, since it must not be connected to the unbounded component on the left. Since each component is strictly to the right of the $(-\alpha_2 - 1)$-th column, there are exactly $h(i - \alpha_1, -\alpha_2 - 1)$ many cubes in $I_{i - \alpha_1} + \alpha$ for which there might be a top cube to their right. For each of those, for exactly one $\alpha_3$ the top cube of the corresponding column in $I_i$ is adjacent to its right, and that top cube would be the top cube of its component. Note that since by Borel-fixedness, we have $m_2(I_i)>m_2(I_{i - \alpha_1})$, there is always a corresponding column of $I_i$ to the right of these $h(i-\alpha_1,-\alpha_2-1)$ many cubes in $I_{i - \alpha_1} + \alpha$, and thus our claim is proved. Therefore, the vectors with $\alpha_2 < 0$ produce $\sum_{\alpha_2 = -m_2(I_{i - \alpha_1})}^{-1} h(i - \alpha_1, -\alpha_2 - 1) = l_{i - \alpha_1}$ many components.

        \item[(II)]   $\alpha_2\geq 0$: In this case, we prove that the vectors with $\alpha_2\geq 0$ yield $l_i$ many components. Observe that we must have $\alpha_2 < m_2(I_i)$ to have a bounded-in-2D component, as otherwise $I_i\cap(\tilde I + \alpha) = \varnothing$. We claim that for each $0\leq\alpha_2<m_2(I_i)$, exactly $h(i, \alpha_2)$ many components are produced. For each $I_i$ and $\alpha_2$, define the margin cubes in $I_i$ to be the ones in each of the $j$th columns of $I_i$ for $j\geq\alpha_2$, to be the ones with no cube adjacent to their right (see Figure \ref{Fig:TlessU1}, (2)). There are $h(i, \alpha_2)$ many margin cubes. For each bounded-in-2D component, consider the rightmost cube in its bottom row. We show that a cube is the bottom-rightmost cube of a bounded-in-2D component if and only if it is a margin cube. The `only if' direction is clear. To show the other direction, we specify the component for which a given  margin cube is the bottom-rightmost cube. If, with fixed $\alpha_1$ and $\alpha_2$, there is a cube of $I_{i - \alpha_1} + \alpha$ with the same $x$ and $y$ coordinates as this margin cube, the component is the one for which the top cube of this column of $I_{i - \alpha_1} + \alpha$ is immediately below the margin cube. If there is no cube of $I_{i - \alpha_1}$ with the same $x$ and $y$ coordinates as the margin cube, then our desired component is the one for which the margin cube and the bottom row of $I_{i - \alpha_1} + \alpha$ have the same $z$-coordinate. Therefore, our claim follows. Hence, the vectors with $\alpha_2\geq 0$ produce $\sum_{\alpha_2 = 0}^{m_2(I_i) - 1} h(i, \alpha_2) = l_i$ many components.  
    \end{itemize}

  \begin{figure}[h]
  \subcaptionbox*{(1) $\alpha=(-1,-1,1)$. In the highlighted component, the cube marked with a square is connected to the unbounded component from the left. The cubes marked with circles are connected to either a bounded component in the previous $x$-slice (the first two) or the unbounded component (the backmost one).}[.45\linewidth]{%
    \includegraphics[width=\linewidth]{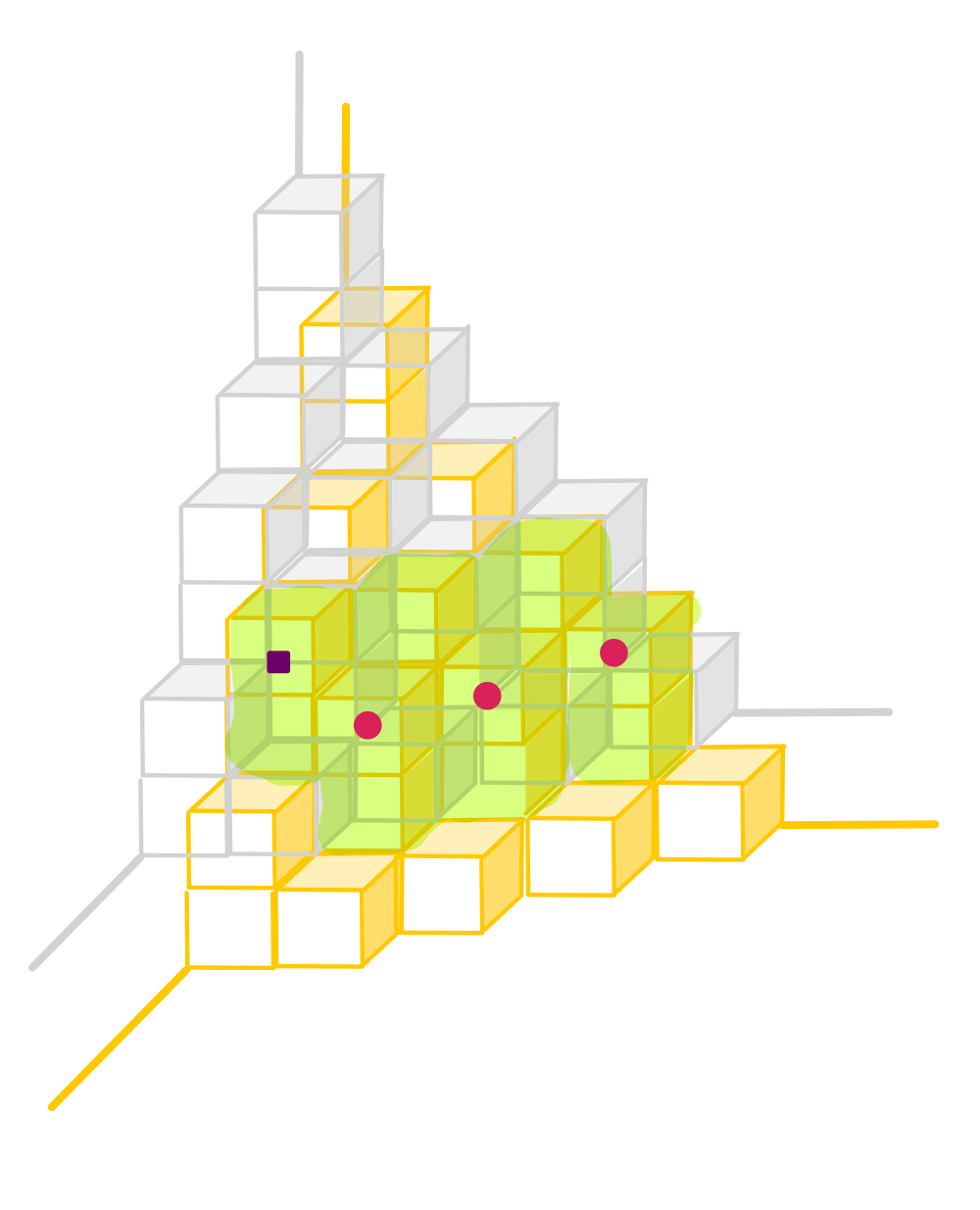}%
  }%
  \hskip7.0ex
  \subcaptionbox*{(2) The margin cubes of $I_0$ with $\alpha_2 = 2$.}[.45\linewidth]{%
    \includegraphics[width=\linewidth]{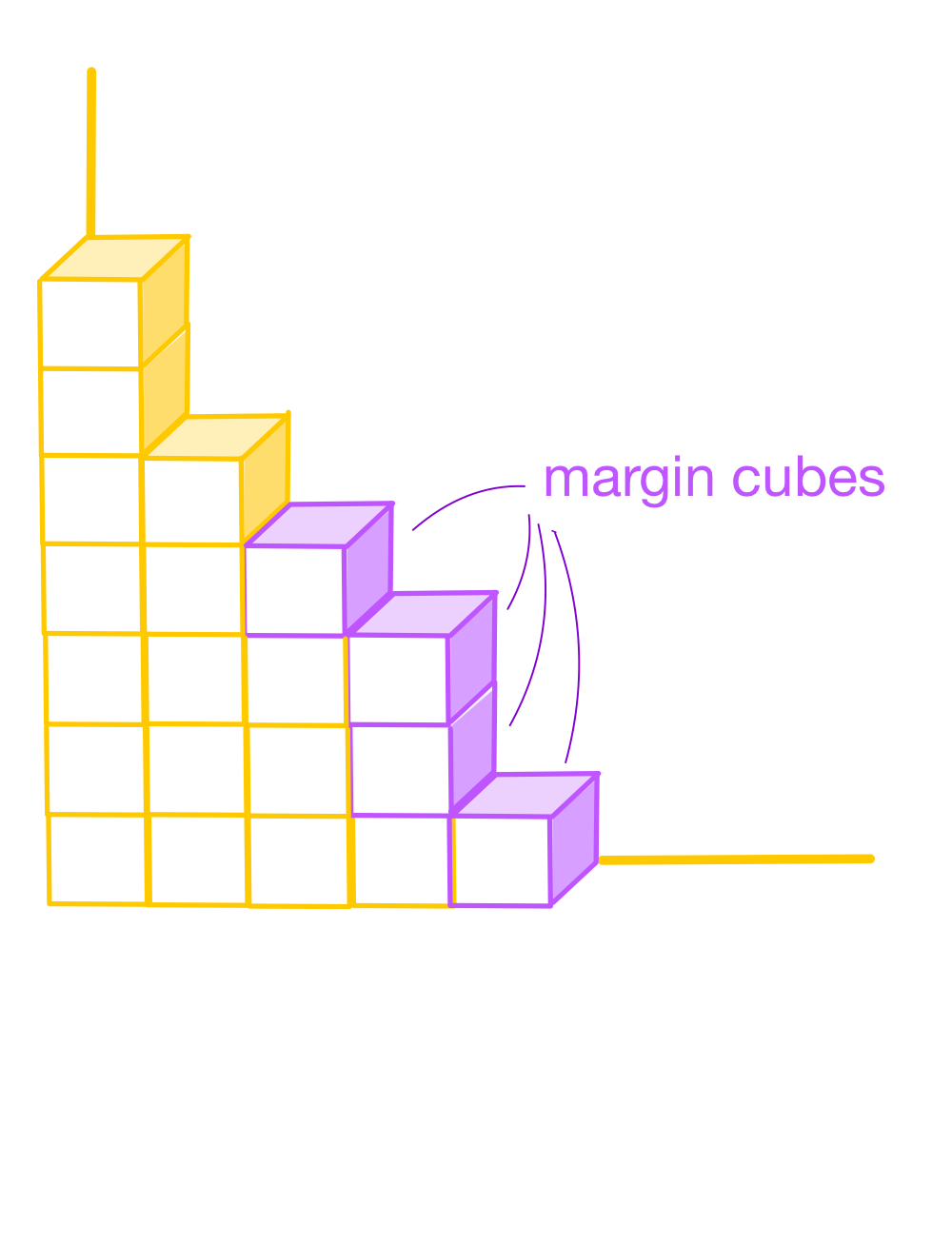}%
  }
  \caption{}
  \label{Fig:TlessU1}
\end{figure}
    
  Finally, we prove that at least $H_{i - \alpha_1 - 1, i}(I)$ many of these $l_i + l_{i - \alpha_1}$ components are connected to either a bounded component in the previous $x$-slice or the unbounded component. The argument is similar to that of the case of the backmost slices within $I_{m_1 + \alpha_1}$: for each $0\leq\alpha_2 < m_2(I_i)$ and $0\leq\alpha_3<h(i, \alpha_2) - h(i - \alpha_1 - 1, 0)$, a bounded-in-2D component produced by the interaction of $I_i$ and $I_{i - \alpha_1} + \alpha$ is connected to the unbounded component if $i - \alpha_1 -1 = 0$, and to a component of the previous $x$-slice, otherwise. This connection occurs above the top cube of the $0$th column of $I_{i - \alpha_1 - 1}$. Therefore, all of our claims regarding this case are proved and we obtain this upper bound for the number of bounded components of $(\tilde I + \alpha)\backslash\tilde I$ for the vectors with $\alpha_1 < 0$:
  \begin{equation}
    \begin{aligned}\label{eq:eq1}
     & \sum_{\alpha_1 = -m_1}^{-1} \Bigg(l_{m_1 + \alpha_1} - H_{m_1 - 1, m_1 + \alpha_1}(I) + \sum_{i = 0}^{m_1 + \alpha_1 - 1} \big(l_i + l_{i-\alpha_1} - H_{i - \alpha_1 - 1, i}(I)\big)\Bigg)\\=&m_1\sum_{j = 0}^{m_1 - 1}l_j - \sum_{0\leq j\leq i\leq m_1-1}H_{i, j}(I).
    \end{aligned}
\end{equation}

   Now, we consider the case $0\leq \alpha_1<m_1$. For each $0\leq \alpha_1<m_1$ and each $\alpha_1\leq i<m_1$, we claim that there are at most $l_i + l_{i - \alpha_1} - S_{i, i - \alpha_1}(I)$ many bounded components with the backmost slice within $I_i$. The proof is similar to the proofs of (I) and (II) above. For the vectors with $0\leq\alpha_2$, the exact proof of (II) works: these vectors produce $l_i$ many bounded-in-2D components. For the vectors with $0 > \alpha_2$, there is a slight difference with (I): we do not have the inequality $m_2(I_i) > m_2(I_{i - \alpha_1})$. We show that for these vectors, at most $l_{i - \alpha_1} - S_{i, i - \alpha_1}(I)$ many bounded-in-2D components are produced. The proof is similar to that of case I, with the consideration that since in this case we have $m_2(I_i)\leq m_2(I_{i - \alpha_1})$, there is a corresponding column of $I_i$ to the right of each column $j$ of $I_{i - \alpha_1} + \alpha$ with $-\alpha_2 - 1\leq j\leq\min(-\alpha_2 + m_2(I_i) - 2, m_2(I_{i - \alpha_1}) - 1)$. So, the number of components produced by the vectors with $0 > \alpha_2$ has the upper bound
    \begin{align*}
      &\sum_{\alpha_2 = -m_2(I_{i - \alpha_1})}^{-m_2(I_{i - \alpha_1}) + m_2(I_i) - 1} \big(h(i - \alpha_1, -\alpha_2 - 1)\big) \\+
      &\sum_{\alpha_2 = -m_2(I_{i - \alpha_1}) + m_2(I_i)}^{-1} \big(h(i - \alpha_1, -\alpha_2 - 1) - h(i - \alpha_1, -\alpha_2 - 1 + m_2(I_i)\big)= l_{i - \alpha_1} - S_{i, i - \alpha_1}(I),
    \end{align*}
    by definition of $S_{i, i - \alpha_1}(I)$. This is an upper bound and not necessarily an exact count, because, unlike the case  (I) above, the counted `components' might be connected to the unbounded component from the below of $I_i$. Overall, for the vector with $\alpha_1\geq0$, the number of produced components has the upper bound
    \begin{align}\label{eq: eq2}
      \sum_{\alpha_1 = 0}^{m_1 - 1} \sum_{i = \alpha_1}^{m_1 - 1} \big(l_i + l_{i - \alpha_1} - S_{i, i - \alpha_1}(I)\big) = (m_1 + 1)\sum_{j = 0}^{m_1 - 1}l_j - \sum_{0\leq j\leq i\leq m_1-1} S_{i, j}(I).
    \end{align}

    Adding up the upper bounds for cases $\alpha_1 < 0$ and $\alpha_1\geq0$ (in \eqref{eq:eq1} and \eqref{eq: eq2}), we get
    \begin{align*}
        m_1\sum_{j = 0}^{m_1 - 1}l_j - \sum_{0\leq j\leq i\leq m_1-1}H_{i, j}(I) + (m_1 + 1)\sum_{j = 0}^{m_1 - 1}l_j - \sum_{0\leq j\leq i\leq m_1-1} S_{i, j}(I)\\=(2m_1 + 1)l - \sum_{0\leq j\leq i\leq m_1-1}H_{i, j}(I) - \sum_{0\leq j\leq i\leq m_1-1} S_{i, j}(I) = U_{m_1}(I).
    \end{align*}

\end{proof}

\section{Canonical Configurations}\label{Section: CanonicalConfig}

  For a fixed degree ${k+3\choose 3}-{b\choose 2} \leq l<{k+3\choose 3}-{b-1\choose 2}$ with $2\leq b\leq k+2$, we construct a canonical ideal $\mathfrak{J}_l$ with $m_1(\mathfrak{J}_l)=k$, such that $G(\FJ_l)={k+3\choose 2}-b$  and $\#\FC(\FJ_l)={k+2\choose 2}-b+1$. Then, we show that the dimension of the Hilbert scheme at $[\FJ_l]$ is at least
  \begin{align*}
     V(l):
    =&G(\FJ_l)\cdot \#\FC(\FJ_l)+\binom{k+3-b}{2}+3\Big(l-{k+3\choose 3}+{b \choose 2}\Big)\\
     =&\Big({k+3\choose 2}-b\Big)\cdot\Big({k+2\choose 2}-b+1\Big)+\binom{k+3-b}{2}+3\Big(l-{k+3\choose 3}+{b \choose 2}\Big).
  \end{align*}

  The purpose of constructing such a configuration is to show that for a fixed degree $l$ and each $1\leq m_1<k$, there is a canonical point $[\FJ_l]$ in $\Hilb^l(\BA^3)$, with $T(\mathfrak{J}_l)$ strictly greater than a weak upper bound $W_{m_1}(l)$, which will be introduced in the following section.

  \begin{definition}[Canonical configuration/ideal $\FJ_l$] \label{Def: canonicalConfig}Starting from the tetrahedron configuration of degree $l={k+2\choose 3}$, we can increase the degree by $1$ by adding an additional cube on the top of the tetrahedron. Then, we add the next cube on the top of the next column on the left in the 0th slice, and continue this until we obtain the ideal $(y,z)^{k+1}$ as the 0th slice. At this stage, $l={k+2 \choose 3}+k+1$. Then, we complete the 1st slice similarly by adding a cube on the top of the columns each time. We can continue this until we complete the $(k-1)$-th slice. At this stage, the degree is ${k+3 \choose 3}-1$. Each of these configurations is called the \emph{canonical configuration} or the \emph{canonical ideal} of the corresponding degree $l$. See Figure \ref{Fig: CanonicalConfig}. 
                 \begin{figure}[H]

  \subcaptionbox*{}[.98\linewidth]{    \includegraphics[width=\linewidth]{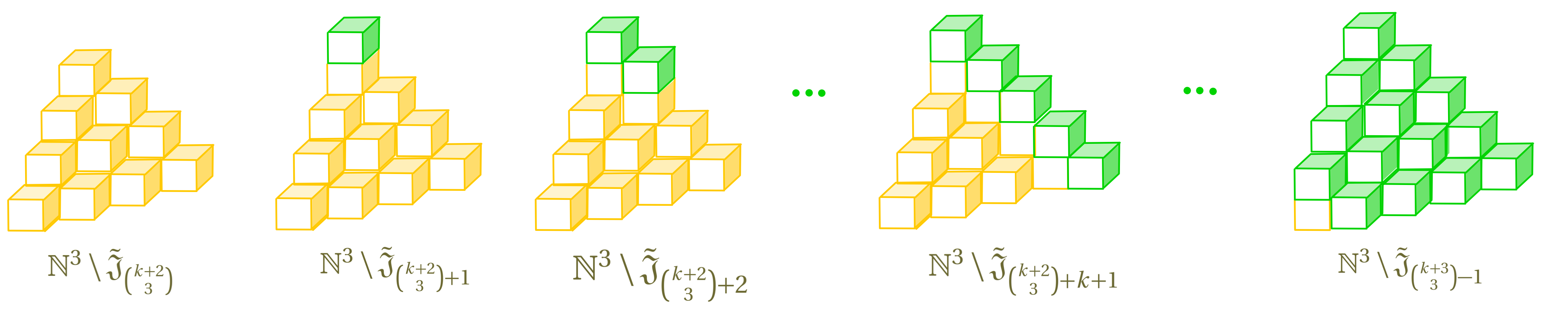}}
  \caption{Constructing the canonical configuration/ideal of each degree ${k+2 \choose 3}\leq l<{k+3 \choose 3}$ from $\mathrm{m}^k$.}
  \label{Fig: CanonicalConfig}

\end{figure}

  \end{definition}

  \begin{lemma}\label{lem: VLessT}
      For the canonical configuration $\FJ_l$ of degree $l$, we have $V(l)\leq T(\FJ_l)$.
  \end{lemma}
            \begin{proof} First, we show that $G(\FJ_l)
     ={k+3\choose 2}-b$ and $\#\FC(\FJ_l)={k+2\choose 2}-b+1$. 

     \begin{itemize}
         \item $G(\FJ_l)
     ={k+3\choose 2}-b$: For the tetrahedron of size $k+1$, i.e., $\BN^3\setminus\widetilde{\mathrm{m}^{k+1}}$, the number of generators is ${k+3\choose 2}$. We now proceed by induction on $b$. For $b=2$, we remove the front cube from the tetrahedron $\BN^3\setminus\widetilde{\mathrm{m}^{k+1}}$, and the number of generators is reduced by $2$.  Suppose that the claim holds for $b=j$ and we want to deduce it for $b=j+1$.  By the induction hypothesis, for ${k+3\choose 3}-{j\choose 2} \leq l<{k+3\choose 3}-{j-1\choose 2}$, we have  $G(\FJ_l)
     ={k+3\choose 2}-j$. Now, to go from $j$ to $j+1$, we remove the rightmost cube in the bottom row of the slice $(\FJ_l)_{k+1-j}$, which reduces the number of generators by one. Note that while $b$ is fixed, the number of generators does not change by increasing $l$. So, the claim follows. 

     \item $\#\FC(\FJ_l)={k+2\choose 2}-b+1$:  Similarly, for the tetrahedron of size $k+1$, we have $\#\FC(\FJ_l)={k+2\choose 2}$.  Proceeding by induction on $b$, for the base case, $b=2$, we remove the front cube from the tetrahedron $\BN^3\setminus\widetilde{\mathrm{m}^{k+1}}$, and $\#\FC(\FJ_l)$ reduces by $1$. The rest of the proof is similar to the previous case.
     \end{itemize}

            Now, we have the following claims.

             \emph{Claim 1.} For each $2\leq b \leq k+2$, for $l = {k+3\choose 3}-{b\choose 2}$, we have
\begin{align}\label{eq:FormulaT2}
    T(\FJ_{l})\geq G(\FJ_{l})\cdot\#\FC(\FJ_{l})+\binom{k+3-b}{2}=V(l).
\end{align}
        
              \begin{proof}[Proof of Claim 1] From Corollary \ref{Cor: LowerBoundOnT}, we have $T(\FJ_l)\geq G(\FJ_l)\cdot \#\FC(\FJ_l)$. We want to introduce $\binom{k+3-b}{2}$ many bounded components of size larger than one (so they will be different from the components of size one corresponding to the hotspot).
              
                  If $b=k+2$, then ${k+3-b\choose 2}=0$; hence, we assume $b<k+2$. For $0\leq j \leq k+1-b$ and $1\leq s \leq j+1$, define $R^{s}_{k+1-b-j}$ to be the $s$th rightmost cube of vertical distance $2$ in $(\FJ_l)_{k+1-b-j}$. Let $g$ be the generator of $\FJ_l$ that is immediately in front of the cube $R^1_{k+1-b}$
                  (see Figure \ref{Fig: CanonicalConfigTangentChange2}, left). The number of the cubes $R^{s}_{k+1-b-j}$ is $\sum_{j=0}^{k+1-b} (j+1)=\sum_{j'=1}^{k+2-b} j'={k+3-b\choose 2}$.

                Then, for any fixed $b$, it suffices to show that if $\alpha$ is a vector from $g$ to $R^{s}_{k+1-b-j}$ for any $s,j$, then $\alpha$ produces a bounded component. We observe that these ${k+3-b\choose 2}$ components are distinct.
                For such $\alpha$, which is of the form $(\alpha_1,\alpha_2,\alpha_3)$, there is a bounded component of $(\tilde \FJ_l+\alpha)\setminus \tilde \FJ_l$ consisting of $2b-1$ cubes. The argument for boundedness is as follows: since   $\alpha_1$ is between $-x(g)$ and $-1$, 
and $\alpha_2,\alpha_3\geq0$, together with having $\alpha_2+\alpha_3=-\alpha_1-1$, and due to the fact that the union of the slices $(\FJ_l)_t$ for $t<x(g)$ is a piece of a tetrahedron, moving towards $\alpha$ induces a connected component. This component is bounded (1) from the left by the left plane of $\tilde \FJ_l+\alpha$; (2) from the below by the bottom plane of $\tilde \FJ_l+\alpha$; (3) from front by $\tilde\FJ_l$; (4) from the back by the triangular slice $(\FJ_l)_{k+1-b}+\alpha$ which lies immediately behind the corresponding triangular slice $(\FJ_l)_{k+1-b-j}$ (see Figure \ref{Fig: Bound}). The latter is  because $R^{s}_{k+1-b-j}$ is of vertical distance $2$, and after moving from $g$ to $R^{s}_{k+1-b-j}$, the top cubes of $\FJ_l$ and the top cubes of $\FJ_l+\alpha$ have the same $z$-coordinate in pairs; the pairs that share the same $y$-coordinate. 

Moreover, the produced component consists of $b-1$ columns of two cubes and one column of one cube (the rightmost cube); adding them up, we get $2b-1$ cubes.

                                              \begin{figure}[h]

  \subcaptionbox*{}[.85\linewidth]{
    \includegraphics[width=\linewidth]{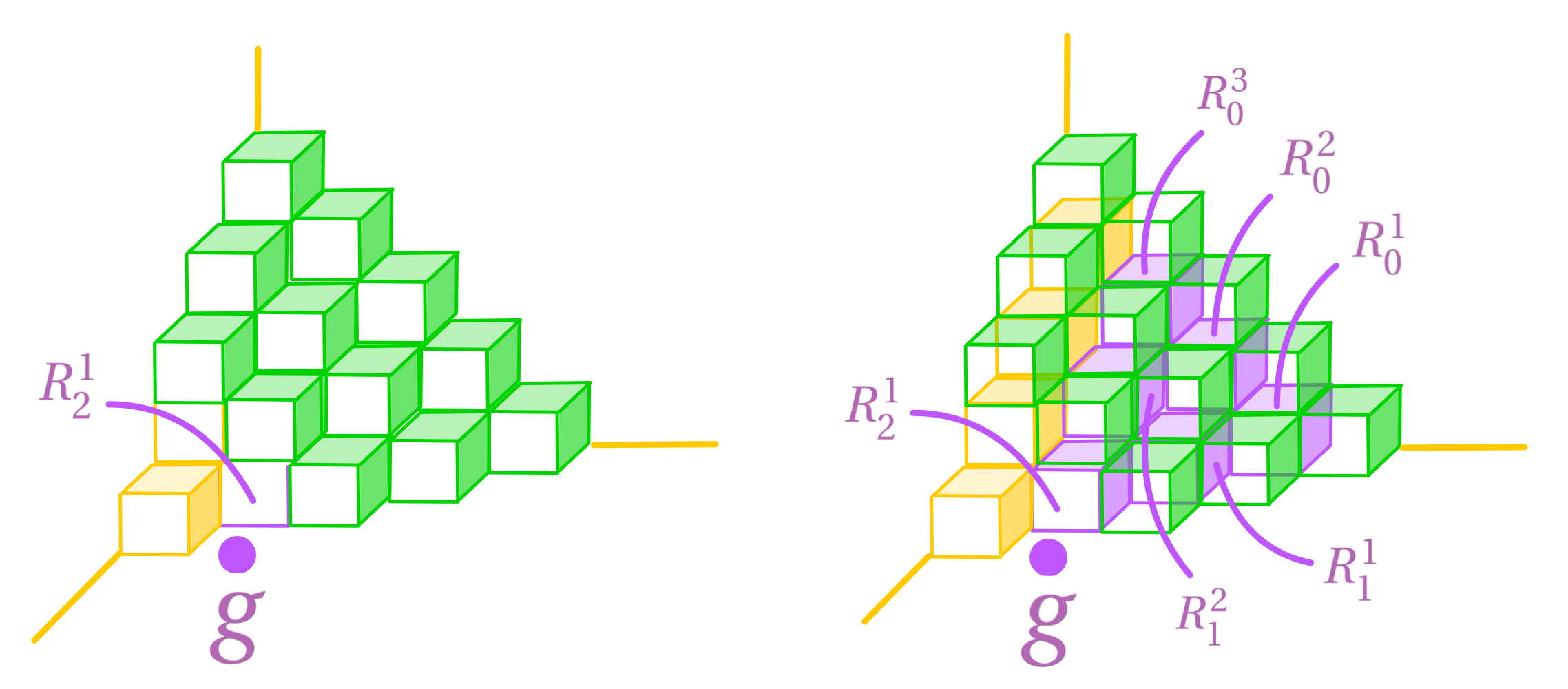}
  }
  \caption{The canonical configuration/ideal $\FJ_{32}$, together with the generator $g$ and the cubes $R^{s}_{k+1-b-j}$.}
  \label{Fig: CanonicalConfigTangentChange2}
\end{figure}

                                              \begin{figure}[h]

  \subcaptionbox*{}[.17\linewidth]{
    \includegraphics[width=\linewidth]{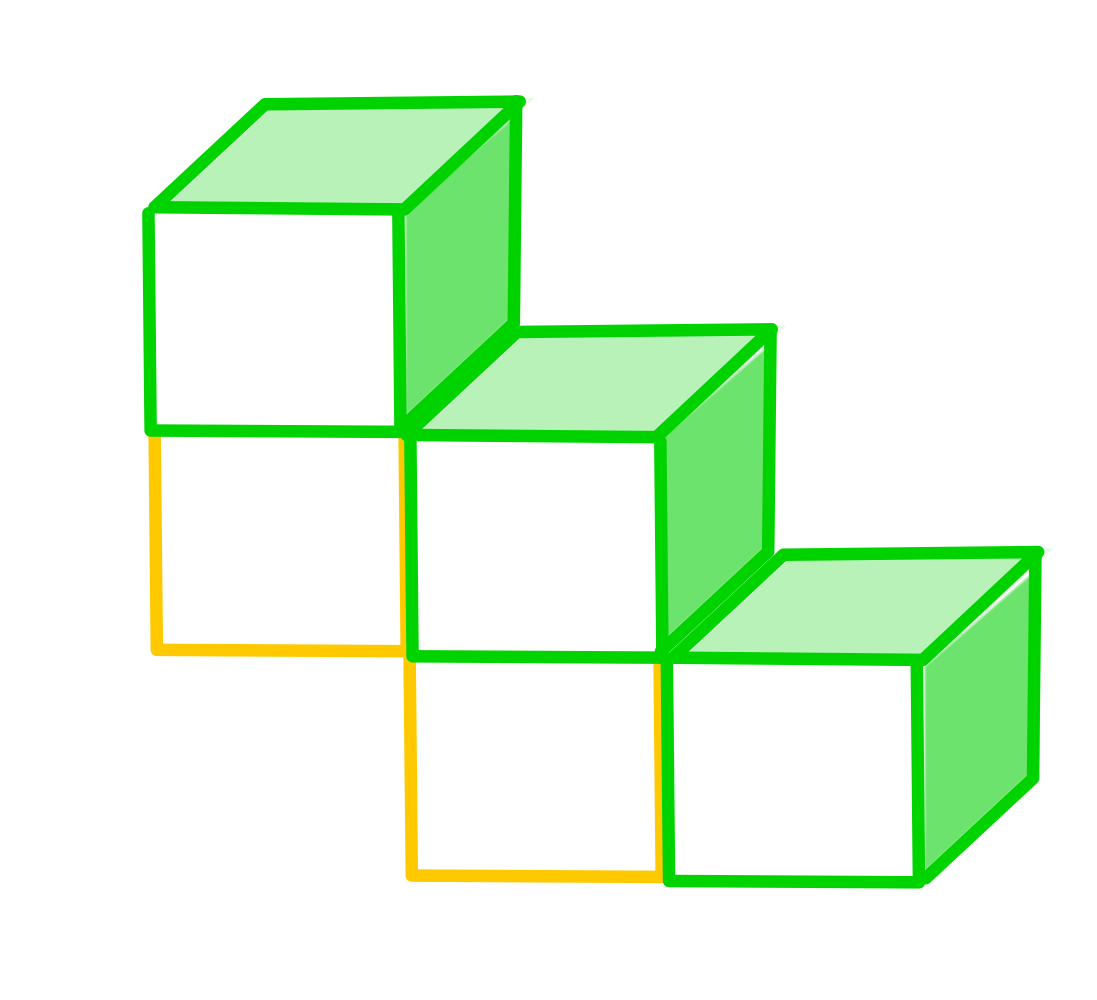}
  }
  \caption{The bounded component of size $2b-1$ in $(\tilde \FJ_l+\alpha)\setminus \tilde \FJ_l$.}
  \label{Fig: BoundedComponent}

\end{figure}

\begin{figure}[h]
  \subcaptionbox*{(1) $\BN^3\setminus \tilde \FJ_l$ and its copy to move along $\alpha$.}[.5\linewidth]{%
    \includegraphics[width=\linewidth]{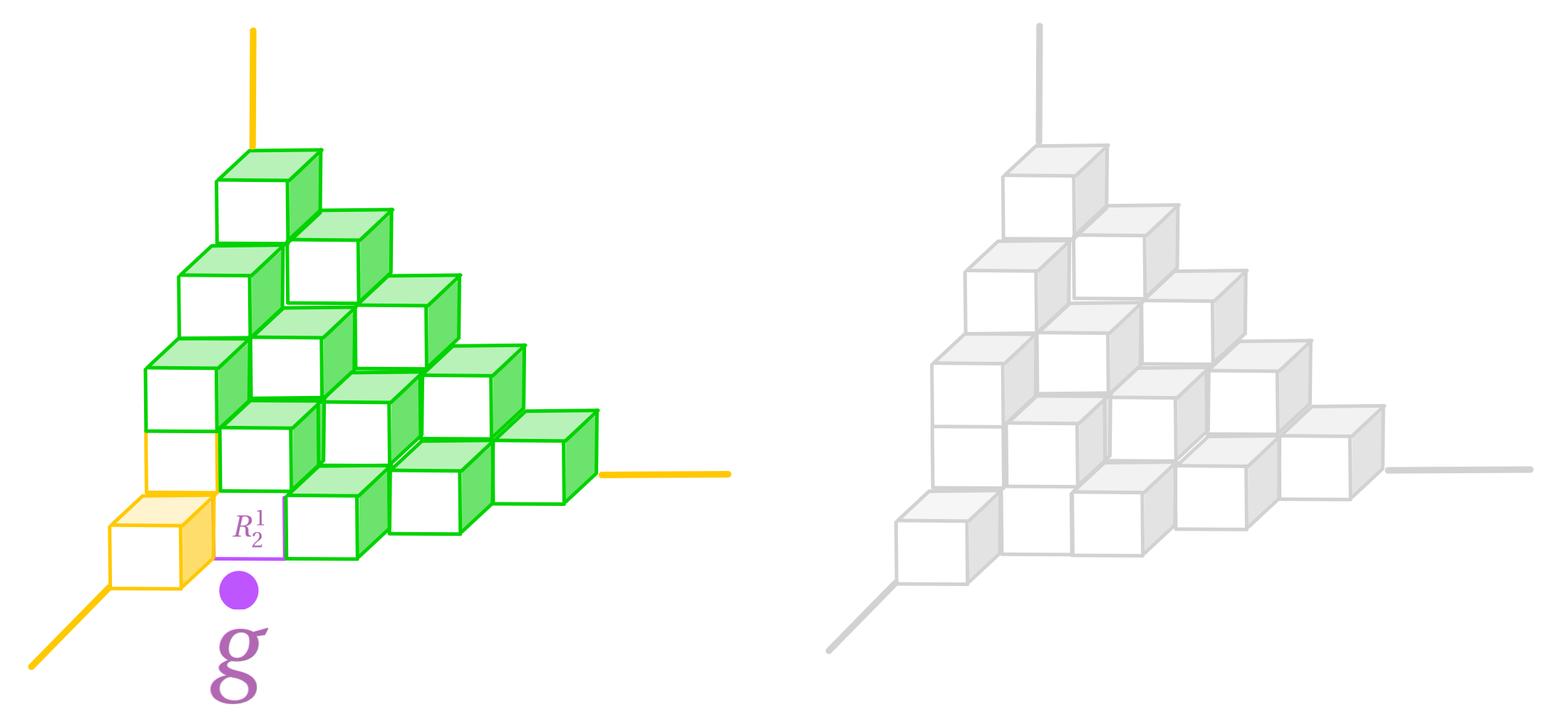}%
  }%
  \hskip7.0ex
  \subcaptionbox*{(2) local picture of $(\tilde \FJ_l+\alpha)\setminus \tilde \FJ_l$. The corresponding bounded component is highlighted.}[.39\linewidth]{%
    \includegraphics[width=\linewidth]{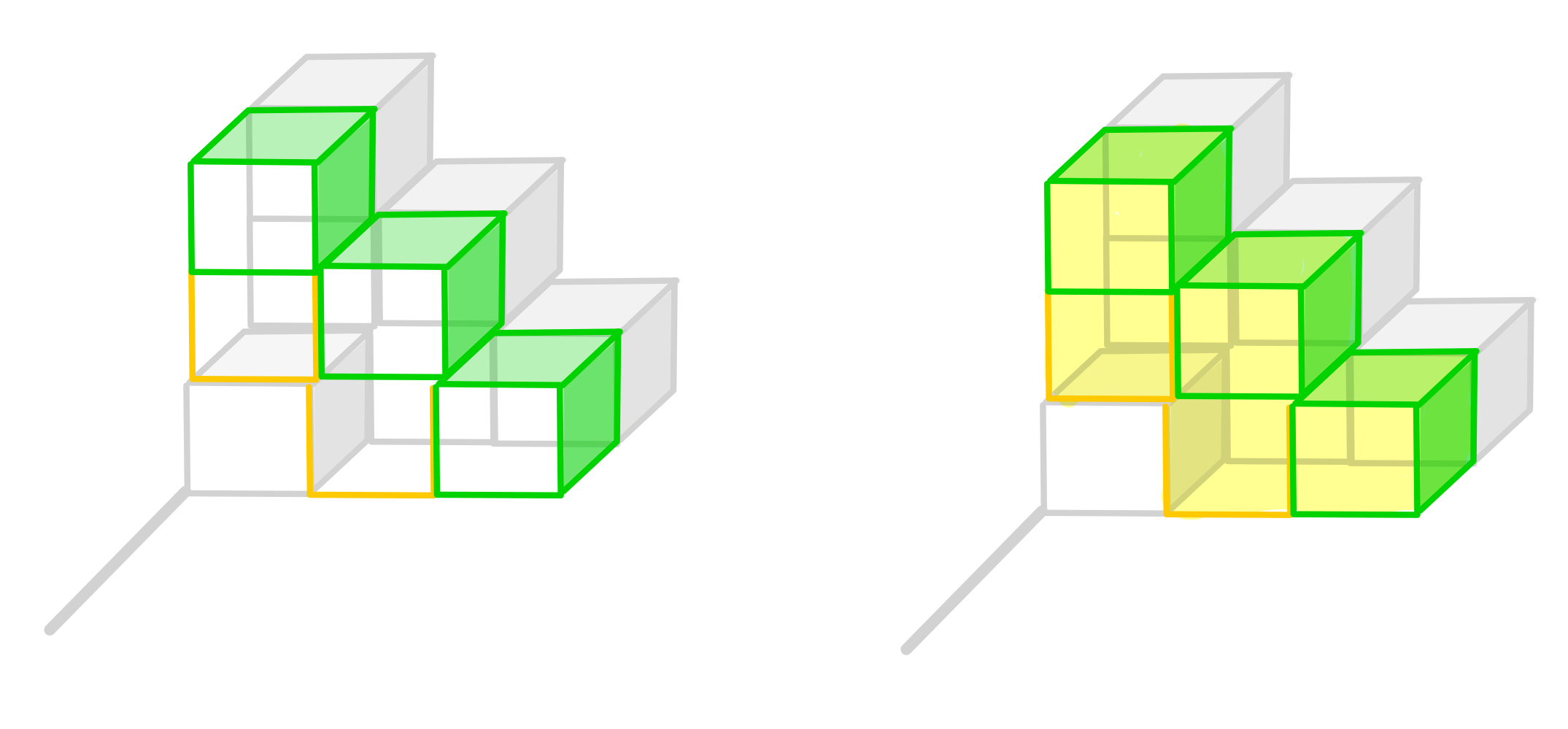}%
  }
  \caption{}
  \label{Fig: Bound}
\end{figure}
               \end{proof}   
              
                      \emph{Claim 2.} For a fixed degree ${k+3\choose 3}-{b\choose 2} < l<{k+3\choose 3}-{b-1\choose 2}$, where $2\leq b \leq k+2$, we have 
            \begin{align}
                T(\FJ_l)\geq G(\FJ_l)\cdot \#\FC(\FJ_l)+\binom{k+3-b}{2}+3\Big(l-{k+3\choose 3}+{b \choose 2}\Big)=V(l).
            \end{align}
            
            \begin{proof}[Proof of Claim 2]
                Let $\FJ_l=(x^{m_1},y^{m_2},z^{m_3}, \text{mixed monomial generators})$. From Corollary \ref{Cor: LowerBoundOnT}, we have $T(\FJ_l)\geq G(\FJ_l)\cdot \#\FC(\FJ_l)$. Let $g$ and $R_{k+1-b-j}^s$ be as in the proof of Claim 1 (in fact, $g$ is the rightmost generator of the only non-triangular slice). A similar argument as in the previous claim implies that there are $\binom{k+3-b}{2}$ additional bounded components given by the vectors from $g$ to $R_{k+1-b-j}^{s}$ (this is because the position of all $R^{s}_{k+1-b-j}$ is to the right of $g$, and this region is unchanged). Note that the only difference here is that the size of the produced bounded component might be smaller than $2b-1$, but it is still an odd number, strictly greater than $1$. This is due to the face that in this case, as in the proof of Claim 1, the component consists of a non-zero number of columns of size $2$ plus a single cube to their right in $\BN^3\setminus \tilde \FJ_l$. We want to construct further $3\Big(l-{k+3\choose 3}+{b \choose 2}\Big)$ many bounded components as follows. For each $1\leq t \leq l-{k+3\choose 3}+{b \choose 2}$, let us denote by $F^t$ the $(l-{k+3\choose 3}+{b \choose 2}+1-t)$-th leftmost cube  of vertical distance $2$ in the same slice as of $g$ in $\BN^3\setminus\tilde\FJ_l$.
            We want to show that, associated to each of the three specific choices of $\alpha$ and for each $F^t$, there is a bounded component of $(\tilde \FJ_l+\alpha)\setminus\tilde\FJ_l$ containing $F^t$. Let $g'$ be the generator of $\FJ_l$ immediately to the right of $F^1$ (note that this may or may not coincide with $g$). Also, let $g''$ be the generator of $\FJ_l$ in front of the only cube in the $xz$-plane of vertical distance $2$. Moreover, let $g'''=z^{m_3}$ be the top generator of $\FJ_l$. We then construct bounded components along vectors from $g',g''$, and $g'''$ to each $F^t$, and since there are $l-{k+3\choose 3}+{b \choose 2}$ many $F^t$'s, this will yield the desired result. For each $1\leq t\leq l-{k+3\choose 3}+{b \choose 2}$, we have the following cases. In the following argument, Figures \ref{Fig: Ex4}, \ref{Fig: Ex5}, and \ref{Fig: Ex6},
            will be helpful.

            \begin{itemize}
                \item Components associated to $\alpha=(\alpha_1,\alpha_2,\alpha_3)$ from $g'$ to $F^t$: The bounded component produced by this $\alpha$ is of size at least two (as it  contains $F^t$ and the hotspot cube above it). 
               Now we proceed with the proof of boundedness of the produced component. In this case, we have $\alpha=(0,\alpha_2,\alpha_3)$, for $\alpha_2<0$ and $\alpha_3\geq0$. The component is contained in the slice $(\FJ_l)_{k+2-b}$, and is (1) bounded from the left by the two cubes in $(\BN^3\setminus\tilde \FJ_l)+\alpha$ which are the shifted copies of $F^1$ and the cube above it; (2) bounded from below by $(\BN^3\setminus\tilde\FJ_l)+\alpha$ and the bottom plane of $\tilde \FJ_l +\alpha$  (the boundedness is due to $\alpha_3$ being non-negative); (3) bounded from the back  by the shifted slice $(\FJ_l)_{k+1-b}+\alpha$ for $b<k+2$, and by the back plane of $\tilde\FJ_l+\alpha$ for $b=k+2$. See Figure \ref{Fig: Ex6}, left.

               Since here $\alpha_1=0$, the produced bounded component is contained in the same slice as $g'$, i.e., the slice $(\FJ_l)_{k+2-b}$. We note that since the previously constructed bounded components from $g$ are contained in the previous $x$-slices, the components cannot coincide.  Also, since the size of this component is at least two, it cannot be the same as the components corresponding to the hotspot. 

                 \item  Components associated to $\alpha=(\alpha_1,\alpha_2,\alpha_3)$ from $g''$ to $F^t$: The bounded component produced by this $\alpha$ has size $2t$ (as it consists of the cubes $F^i$, for $1\leq i \leq t$,  and the hotspot cubes above them). Thus, due to parity, and the fact that  $\alpha_1=-1$ in this case, it will be different from the already constructed components. 
                 The component is contained in the slice $(\FJ_l)_{k+2-b}$, and is (1) bounded from the back by the slice $(\FJ_l)_{k+2-b}+\alpha$; 
                 (2) bounded from below by $(\BN^3\setminus\tilde\FJ_l)+\alpha$ and the bottom plane of $\tilde \FJ_l +\alpha$;  (3) bounded from the left by the left plane of $\tilde \FJ_l +\alpha$. See Figure \ref{Fig: Ex6}, middle.

                  \item  Components associated to $\alpha=(\alpha_1,\alpha_2,\alpha_3)$ from $g'''$ to $F^t$: Similarly, the bounded component produced by this $\alpha$ has size $2t$. In this case, the vector $\alpha$ is equal to $(\alpha_1,\alpha_2,\alpha_3)$ for $\alpha_1\geq0$ and $\alpha_3<0$. So, combining these with parity, we see that the new component is different from the already constructed ones. Moreover, the component is contained in the slice $(\FJ_l)_{k+2-b}$, and is (1) bounded from the left by the left plane of $\tilde \FJ_l +\alpha$; (2) bounded from below by  $(\FJ_l)_0 +\alpha$; (3) bounded from the back  by the back plane of $\tilde \FJ_l +\alpha$. See Figure \ref{Fig: Ex6}, right.
            \end{itemize}

       \begin{figure}[h]

  \subcaptionbox*{}[.99\linewidth]{
    \includegraphics[width=\linewidth]{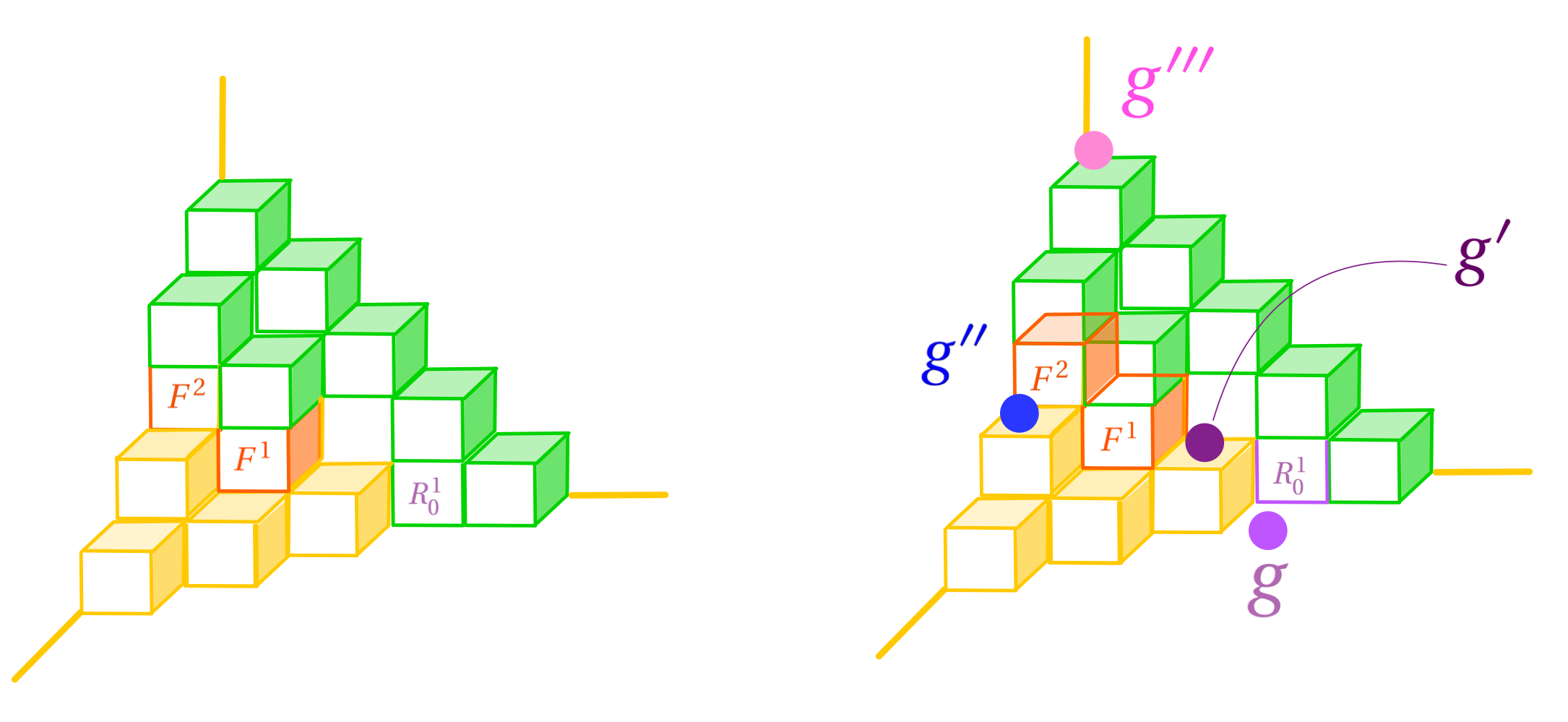}
  }
  \caption{$\BN^3\setminus\tilde \FJ_l$ and the corresponding generators.}
  \label{Fig: Ex4}

\end{figure}
          
      \begin{figure}[h]

  \subcaptionbox*{}[.99\linewidth]{
    \includegraphics[width=\linewidth]{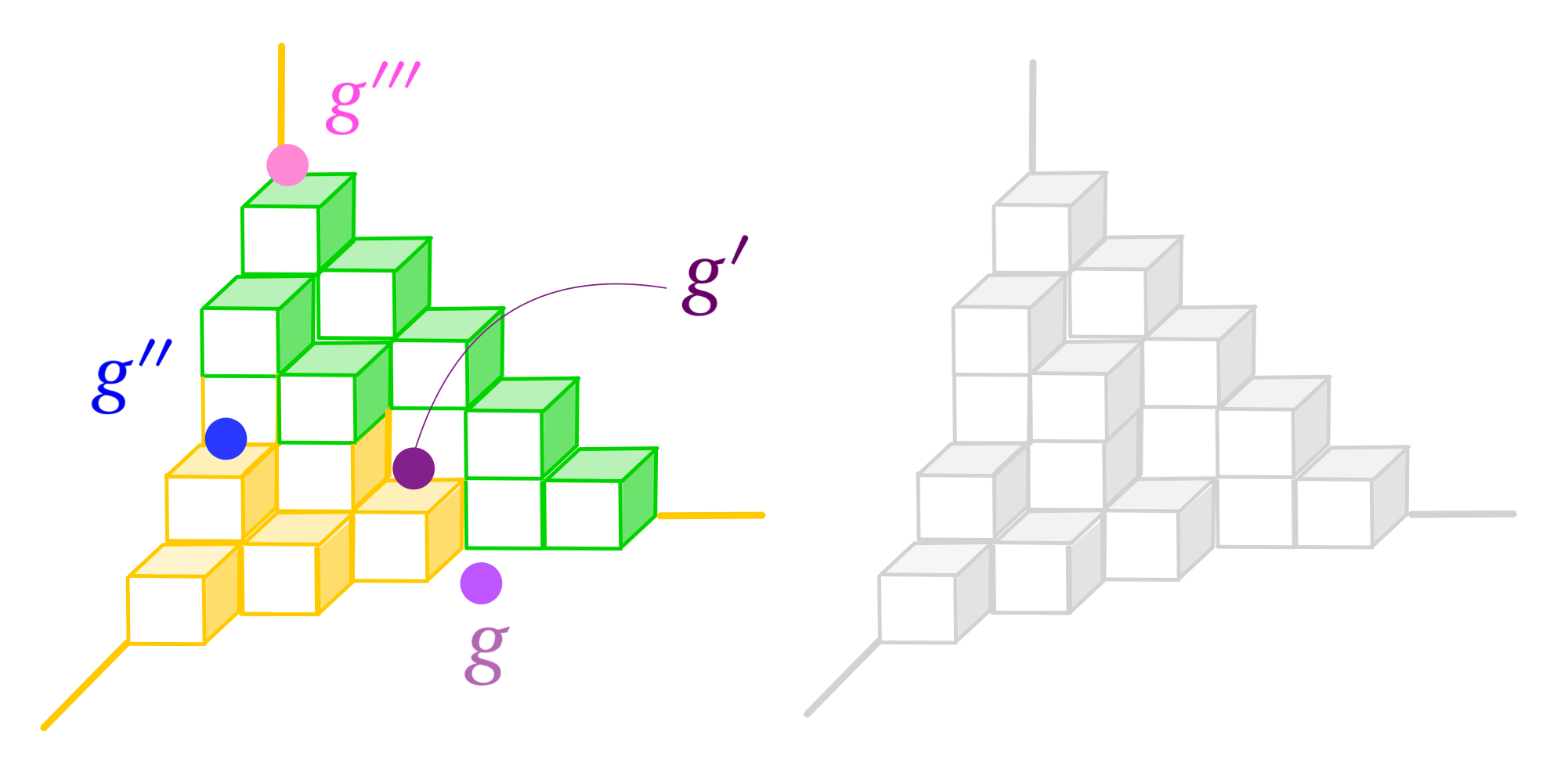}
}
  \caption{$\BN^3\setminus\tilde \FJ_l$ and a copy of it to move along $\alpha$.}
  \label{Fig: Ex5}

\end{figure}

      \begin{figure}[h]

  \subcaptionbox*{}[.99\linewidth]{
    \includegraphics[width=\linewidth]{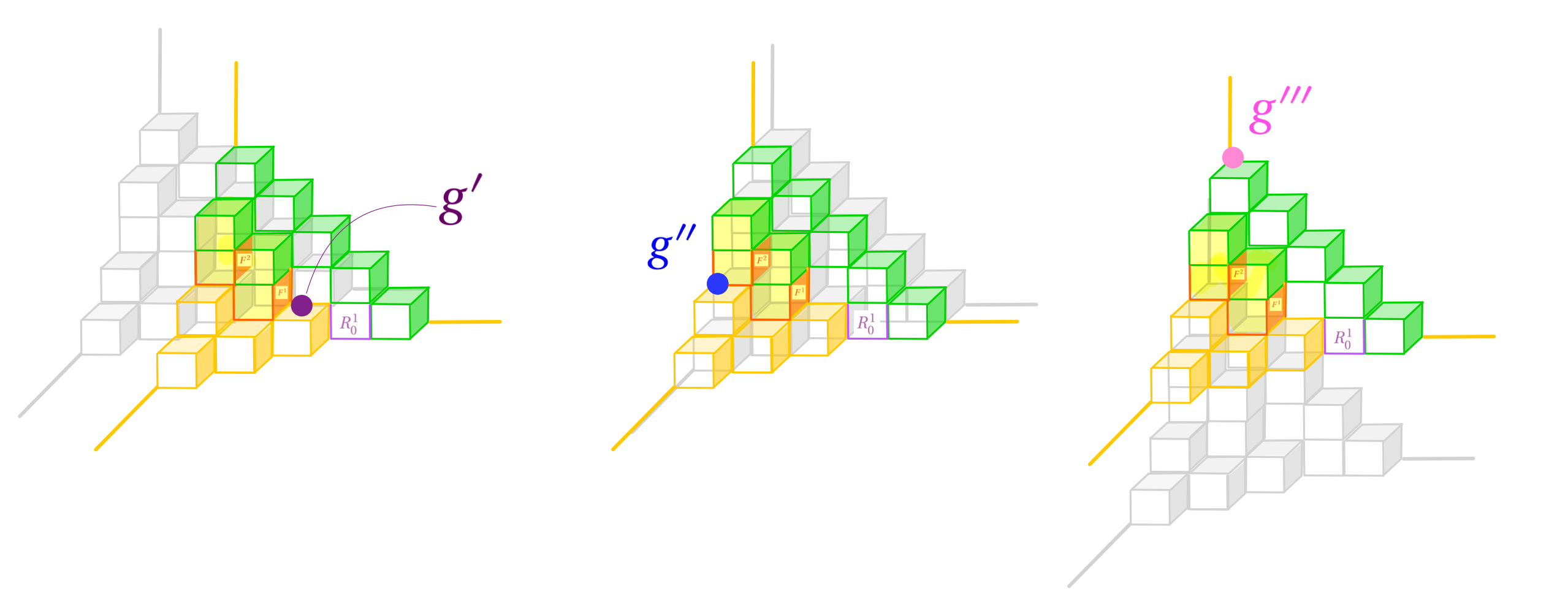}
  }
  \caption{$(\tilde \FJ_l+\alpha)\setminus \tilde \FJ_l$, where the source of $\alpha$ is either $g'$, $g''$ or $g'''$, and the target is $F^t$. The corresponding bounded component is highlighted.}
  \label{Fig: Ex6}

  \end{figure}


\end{proof}

Now, from Claim 1 and Claim 2, we obtain the desired inequality $T(\FJ_l)\geq V(l)$.
            \end{proof}     
            \begin{remark}
 Indeed, we have equality in Lemma \ref{lem: VLessT}, i.e., $T(\FJ_l)=V(l)$, for any $l$. However, since only the inequality is sufficient for the purpose of our main result, we skip the proof of equality.
            \end{remark}

  \section{A weak upper bound}
  In this section, as an intermediate step in our argument, we define an upper bound on $U_{m_1}(I)$, which will also be a lower bound on $V(l)$. 
  
  For an integer $l$ of types \eqref{Type 1}-\eqref{Type 3}, and  for a positive integer $m_1$, we define 
   \begin{align}\label{eq: WeakUpperBound}
       W_{m_1}(l):=(2m_1+1)l-2{m_1+2\choose 4}.
   \end{align}

   For any zero-dimensional Borel-fixed ideal $I$ of degree $l$ of types \eqref{Type 1}-\eqref{Type 3} with $m_1(I)=m_1$, we first show that $W_{m_1}(l)$ is an upper bound on $U_{m_1}(I)$.

   \begin{prop} \label{prop: ULessW}
      For any zero-dimensional Borel-fixed ideal $I$ of degree $l$ of  types \eqref{Type 1}-\eqref{Type 3} with $m_1(I)=m_1$, we have $ U_{m_1}(I)\leq W_{m_1}(l)$. 
   \end{prop}
   \begin{proof} First, we have the following claim.

    \emph{Claim.} For all $0\leq j\leq i\leq m_1 - 1$, we have $H_{i, j}(I)\geq{i - j + 1\choose2}$ and $S_{i, j}(I)\geq{i - j + 1\choose2}$.
       \begin{proof}[Proof of Claim]
        Recall that $h(i,j)$ is the height of the $j$th column of $I_i$. By Borel-fixedness, we have $h(i, 0)\leq h(i - 1, 1)\leq h(i - 2, 2)\leq\dots\leq h(j, i - j)$, and also for all $r$ with $0\leq r<i - j$, we have $h(j, r)\geq h(j, i - j) + i - j - r$. Therefore,
            \begin{align*}
                H_{i, j}(I)\geq\sum_{r = 0}^{i - j - 1} \big(h(j, r) - h(i, 0)\big)\geq\sum_{r = 0}^{i - j - 1} \big(i - j - r\big) = {i - j + 1\choose 2}.
            \end{align*}
            An identical proof works for $S_{i, j}(I)$.
       \end{proof}

       Using the inequalities obtained in the claim, we have
       \begin{align*}
           U_{m_1}(I) =& (2m_1 + 1)l - \sum_{0\leq j\leq i\leq m_1-1}H_{i, j}(I) - \sum_{0\leq j\leq i\leq m_1-1} S_{i, j}(I)\\\leq& (2m_1 + 1)l - 2\sum_{0\leq j\leq i\leq m_1-1}{i - j + 1\choose2} = (2m_1 + 1)l - 2\sum_{r = 1}^{m_1}(m_1 + 1 - r){r\choose2} = W_{m_1}(l).
       \end{align*}
   \end{proof}

   Next, we show that $W_{m_1}(l)$ increases in $m_1$.

   \begin{prop}\label{prop:WIncreasing} For all $\text{\:}{k+2\choose 3}\leq l<{k+3\choose 3}$ and for $\text{\:}1\leq m_1<k$, we have that $W_{m_1}(l)$ increases in $m_1$.
   \end{prop}
   \begin{proof}
       For $1\leq m_1 \leq k - 1$, noting that $l\geq{k+2\choose 3}$, we have
       \begin{align*}
           W_{m_1 + 1}(l) - W_{m_1}(l) = 2l - 2\big({m_1 + 3\choose4} - {m_1 + 2\choose4}\big) = 2\big(l - {m_1 + 2\choose3}\big)\geq0.
       \end{align*}
   \end{proof}

   Finally, we show that $W_{m_1}(l)$ is a lower bound on $V(l)$.
   
       \begin{prop} \label{Prop: WLessV}
      For all $\text{\:}{k+2\choose 3}\leq l<{k+3\choose 3}$ of types \eqref{Type 1}-\eqref{Type 3} and for $\text{\:}1\leq m_1<k$, we have $W_{m_1}(l)<V(l)$. 
   \end{prop}
   \begin{proof} For such $l$, we want to show
   \begin{align*}
       (2m_1+1)l-2{m_1+2\choose 4}<\Big({k+3\choose 2}-b\Big)\cdot\Big({k+2\choose 2}-b+1\Big)+\binom{k+3-b}{2}+3\Big(l-{k+3\choose 3}+{b \choose 2}\Big).
   \end{align*}
   By proposition \ref{prop:WIncreasing}, it suffices to show this for $m_1 = k -1$.

   \textbf{Case 1.}  $l={k+2\choose 3}+j$,  for $0\leq j\leq (k+3)/2$.

   In this case, $b = k + 2$, since $l \leq {k + 2\choose3} + \frac{k + 3}{2} < {k + 3\choose3} - {k + 1\choose2}$ for $k > 1$. Thus, we have to show
    \begin{align*}
       (2k - 1)l-2{k + 1\choose 4}<\Big({k+3\choose 2}-(k + 2)\Big)\cdot\Big({k+2\choose 2}-(k + 2)+1\Big)+3\Big(l-{k+3\choose 3}+{k + 2 \choose 2}\Big).
   \end{align*}
   Substituting $l = {k + 2\choose3} + j$, expanding and simplifying, this is equivalent to $0<j(-2k+4)+k^2+k$, which, for $k > 2$, is equivalent to
   \begin{align*}
       j<\frac{k^2+k}{2k - 4} = \frac{k + 3}{2} + \frac{3}{k - 2}
   \end{align*}
   which is true since $j\leq(k+3)/2$. For $k = 2$, the desired inequality is equivalent to $0 < 6$, and hence is true.

     \textbf{Case 2.}  $l={k+3\choose 3}-j$,  for $j=1,2,3,4$.
     
For $j = 1 = {2\choose2}$ and $j = 3 = {3\choose2}$, this is confirmed in case 3. For $j = 2$, we have $b = 3$ and therefore we have to show that
     \begin{align*}
       &(2k - 1)\big({k + 3\choose3} - 2\big)-2{k + 1\choose 4}\\<&\Big({k+3\choose 2}-3\Big)\cdot\Big({k+2\choose 2}-2\Big)+\binom{k}{2}+3\Big(\big({k + 3\choose3} - 2\big)-{k+3\choose 3}+3\Big),
   \end{align*}
   which is equivalent to $0 < k^2 + k + 2$, which holds for all positive $k$.

   For $j = 4$, we have $b = 4$ and the desired inequality is equivalent to $0 < 6$, and thus holds.

      \textbf{Case 3.}  $l={k+3\choose 3}-{b\choose 2}$ for $b=2,3,4,k+1,k+2$.

      In this case, we need to check
      
      \begin{align}\label{eq: ineq3}
       (2k-1)\big({k+3\choose 3}-{b\choose 2}\big)-2{k + 1\choose 4}<\Big({k+3\choose 2}-b\Big)\cdot\Big({k+2\choose 2}-b+1\Big)+\binom{k+3-b}{2}.
      \end{align}

For $b=2$, this can be reduced to 

 \begin{align*}
       (2k-1)\big({k+3\choose 3}-1\big)-2{k+1\choose 4}<\Big({k+3\choose 2}-2\Big)\cdot\Big({k+2\choose 2}-1\Big)+\binom{k+1}{2},
   \end{align*}
   which is equivalent to showing that $k^2+2k>0$. But this is always true for $k\geq 1$, so we have the desired inequality in this case.

         For $b=3,4,k+1,k+2$, \eqref{eq: ineq3} can similarly be reduced to $k^2+3k-2>0$, $k>0.5$, $k>-2$, and $k^2+k>0$, respectively, which are always true for $k\geq 1$ in all cases. Therefore, we have the desired result.

   \end{proof}

   \section{Proof of the necessary condition conjecture} We can now complete the proof of the main result.

   \begin{theorem}[Conjecture B for degrees of types \eqref{Type 1}-\eqref{Type 3}]\label{Cor: ConjB} If $I=(x^{m_1},y^{m_2},z^{m_3},\text{mixed terms})$ is a zero-dimensional Borel-fixed ideal in $\BC[x,y,z]$ of degree ${k+2 \choose 3}\leq l<{k+3 \choose 3}$ of  types \eqref{Type 1}-\eqref{Type 3} with maximal singularity in $\Hilb^l(\BA^3)$, then $m_1=k$.
       
   \end{theorem}
   \begin{proof}
This is an immediate conclusion of Lemma \ref{lem: TLessU}, Proposition \ref{prop: ULessW}, Proposition \ref{Prop: WLessV}, Lemma \ref{lem: VLessT}, and the fact that $m_1(\FJ_l)=k$. More precisely, if we assume $m_1(I)<k$, then $T(I)\leq U_{m_1}(I)\leq W_{m_1}(l)<V(l)\leq T(\FJ_l)$; i.e., we have found an ideal $\FJ_l$ of degree $l$ and $m_1(\FJ_l)=k$ such that $T(I) < T(\FJ_l)$.
   \end{proof}

The following corollary was first proved in \cite{MR1}, which  also follows as a special case of the main result in the present paper:
      \begin{corollary}[Brian\c{c}on-Iarrobino Conjecture in 3D] \label{Cor:BIConj3D} The Brian\c{c}on-Iarrobino Conjecture holds in three dimensions: $[\mathrm{m}^k]$ is the unique point of the maximum singularity in $\Hilb^{{k+2\choose 3}}(\BA^3)$.
   \end{corollary}
   \begin{proof}
             From \cite[Lemma 1.7]{Rezaee-23-Conjectures}, $\mathrm{m}^k$ is the only ideal of degree ${k+2\choose 3}$ with $m_1=k$. Hence, by Theorem \ref{Cor: ConjB}, we conclude that $[\mathrm{m}^k]$ is the only point of maximum singularity in $\Hilb^{{k+2\choose 3}}(\BA^3)$.
   \end{proof}

\bibliographystyle{amsplain-nodash}

\bibliography{bib}

\vspace{0.25cm}

\end{document}

%% file: macros.tex
\usepackage[colorinlistoftodos]{todonotes}

    \usepackage[bookmarks=false]{hyperref}
\definecolor{antiquewhite}{rgb}{0.98, 0.92, 0.84}
\definecolor{buff}{rgb}{0.94, 0.86, 0.51}
\definecolor{palecopper}{rgb}{0.85, 0.54, 0.4}
\definecolor{fluorescentyellow}{rgb}{0.8, 1.0, 0.0}
\definecolor{bole}{rgb}{0.47, 0.27, 0.23}

\usepackage{amsmath,float}
\usetikzlibrary{decorations.markings}

\usepackage[vcentermath]{youngtab}
\input xy
\xyoption{all}

\usepackage{pgfkeys}
\usepackage{pgfopts}
\usepackage{ytableau}

\definecolor{cornellred}{rgb}{0.7, 0.11, 0.11}
\definecolor{britishracinggreen}{rgb}{0.0, 0.26, 0.15}
\definecolor{cobalt}{rgb}{0.0, 0.28, 0.67}
\DeclareSymbolFont{usualmathcal}{OMS}{cmsy}{m}{n}
\DeclareSymbolFontAlphabet{\mathcal}{usualmathcal}

\newcommand{\BA}{{\mathbb{A}}}

\newcommand{\BC}{{\mathbb{C}}}

\newcommand{\BN}{{\mathbb{N}}}

\newcommand{\BZ}{{\mathbb{Z}}}

\newcommand{\FC}{{\mathfrak{C}}}

\newcommand{\FJ}{{\mathfrak{J}}}

\DeclareMathOperator{\Hilb}{Hilb}

\DeclareFontFamily{OT1}{rsfs}{}
\DeclareFontShape{OT1}{rsfs}{n}{it}{<-> rsfs10}{}
\DeclareMathAlphabet{\curly}{OT1}{rsfs}{n}{it}
\renewcommand\hom{\mathscr{H}\kern-0.3em\mathit{om}}

\DeclareMathOperator{\lHom}{\mathscr{H}\kern-0.3em\mathit{om}}
\DeclareMathOperator{\RRlHom}{\mathbf{R}\kern-0.025em\mathscr{H}\kern-0.3em\mathit{om}}

\DeclareMathOperator{\lExt}{{\mathscr{E}\kern-0.2em\mathit{xt}}}



\usepackage[all]{xy}
\usepackage{tikz}
\usepackage{tikz-cd}
\usepackage{adjustbox}
\usepackage{rotating}
\usepackage{comment}

\usetikzlibrary{matrix,shapes,intersections,arrows,decorations.pathmorphing}
\tikzset{commutative diagrams/arrow style=math font}
\tikzset{commutative diagrams/.cd,
mysymbol/.style={start anchor=center,end anchor=center,draw=none}}

\tikzset{
shift up/.style={
to path={([yshift=#1]\tikztostart.east) -- ([yshift=#1]\tikztotarget.west) \tikztonodes}
}
}

\theoremstyle{definition}

\newtheorem*{lemma*}{Lemma}
\newtheorem*{theorem*}{Theorem}
\newtheorem*{example*}{Example}
\newtheorem*{fact*}{Fact}
\newtheorem*{notation*}{Notation}
\newtheorem*{definition*}{Definition}
\newtheorem*{prop*}{Proposition}
\newtheorem*{remark*}{Remark}
\newtheorem*{corollary*}{Corollary}

\newtheorem*{conventions*}{Conventions}

\newtheorem{definition}{Definition}[section]

\newtheorem{remark}[definition]{Remark}

\newtheoremstyle{thm} 
        {3mm}
        {3mm}
        {\slshape}
        {0mm}
        {\bfseries}
        {.}
        {1mm}
        {}
\theoremstyle{thm}
\newtheorem{theorem}[definition]{Theorem}
\newtheorem{corollary}[definition]{Corollary}
\newtheorem{lemma}[definition]{Lemma}

\newtheorem{prop}[definition]{Proposition}

\newtheoremstyle{ex} 
        {3mm}
        {3mm}
        {}
        {0mm}
        {\scshape}
        {.}
        {1mm}
        {}
\theoremstyle{ex}

\newtheoremstyle{sol} 
        {3mm}
        {3mm}
        {}
        {0mm}
        {\scshape}
        {.}
        {1mm}
        {}
\theoremstyle{sol}

\usepackage{tikz}
\usepackage{xparse}

\newcount\tableauRow
\newcount\tableauCol

\newenvironment{Tableau}[1]{%
  \tikzpicture[scale=0.5,draw/.append style={thick,black},
                      baseline=(current bounding box.center)]
    \tableauRow=-1.5
    \foreach \Row in {#1} {
       \tableauCol=0.5
       \foreach\k in \Row {
         \draw[thin](\the\tableauCol,\the\tableauRow)rectangle++(1,1);
         \draw[thin](\the\tableauCol,\the\tableauRow)+(0.5,0.5)node{$\k$};

                  \global\advance\tableauCol by 1
       }
       \global\advance\tableauRow by -1
    }
}{\endtikzpicture}

\newtheorem*{Acknowledgments*}{Acknowledgments}